\documentclass[11pt,reqno,final]{amsart}

\usepackage[utf8]{inputenc}

\usepackage[hyperindex=true,bookmarks=true,bookmarksnumbered=true]{hyperref}
\usepackage[theoremdefs]{latexdev}
\usepackage{amsmath}
\usepackage{amssymb}
\usepackage{amsfonts}
\usepackage{graphicx}
\usepackage[ruled]{algorithm}
\usepackage[noend]{algpseudocode}
\usepackage{enumerate}
\usepackage{todonotes}
\usepackage{mathtools}
\mathtoolsset{showonlyrefs}
\usepackage{tikz-cd}
\usepackage{caption,subcaption}

\usepackage{array}
\newcommand{\PreserveBackslash}[1]{\let\temp=\\#1\let\\=\temp}
\newcolumntype{C}[1]{>{\PreserveBackslash\centering}p{#1}}
\newcolumntype{R}[1]{>{\PreserveBackslash\raggedleft}p{#1}}
\newcolumntype{L}[1]{>{\PreserveBackslash\raggedright}p{#1}}

\usepackage{geometry}                
\DeclareMathOperator*{\argmax}{arg\,max}
\newtheorem{claim}{Claim}

\newcommand{\reals}{\mathbb{R}}

\newcommand{\comment}[1]{}
\newcommand{\bsum}[2]{\sum\limits_{#1}^{#2}}
\newcommand{\N}{\mathbb{N}}
\newcommand{\R}{\mathbb{R}}

\newcommand{\DiffPL}{\Gamma_{\operatorname{PL}}}
\newcommand{\DiffSmooth}{\Gamma_{\infty}}
\newcommand{\DiffLip}{\Gamma_{\operatorname{Lip}}}
\newcommand{\Imm}{\operatorname{Imm}}
\newcommand{\ImmLip}{\operatorname{Imm}_{\operatorname{Lip}}}
\newcommand{\Diff}{\operatorname{Diff}}

\usepackage[normalem]{ulem}

\makeatletter
\@namedef{subjclassname@2020}{\textup{2020} Mathematics Subject Classification}
\makeatother

\title[SRNF Distance and Unbalanced Optimal Transport]{The Square Root Normal Field Distance and Unbalanced Optimal Transport}
\author{Martin Bauer, Emmanuel Hartman, Eric Klassen}
\thanks{Martin Bauer and Emmanuel Hartman were partially supported by NSF-grants 1912037 and 1953244.}
\begin{document}

\maketitle
\begin{abstract}
This paper explores a novel connection between two areas: shape analysis of surfaces and unbalanced optimal transport. Specifically, we characterize the square root normal field (SRNF) shape distance as the pullback of the Wasserstein-Fisher-Rao (WFR) unbalanced optimal transport distance. In addition we propose a new algorithm for computing the WFR distance and present numerical results that highlight the effectiveness of this algorithm. As a consequence of our results we obtain a precise method for computing the SRNF shape distance directly on piecewise linear surfaces and gain new insights about the degeneracy of this distance.
\end{abstract}

\tableofcontents
\addtocontents{toc}{\protect\setcounter{tocdepth}{1}}

\section{Introduction}
This paper contributes to two different areas: elastic shape analysis (ESA)~\cite{srivastava2016functional} and unbalanced optimal mass transport~\cite{chizat2018scaling,liero2018optimal}. The main results of our article are twofold: first we develop a new algorithm for the numerical computation of the Wasserstein-Fisher-Rao distance~\cite{liero2018optimal,chizat2018interpolating,kondratyev2016new} (a form of unbalanced optimal mass transport), and secondly we establish a connection between these two areas, which in turn allows for the exact computation of the Square Root Normal Field distance, which is a widely used similarity measure in ESA of surfaces; see e.g.~\cite{jermyn2012elastic,laga2017numerical,laga20183d,jermyn2017elastic} and the references therein. 
Before we describe the contributions of the present article in more detail, we will briefly discuss the background of these two fields.

\subsection*{Background:}
In mathematical shape analysis, one is interested in quantifying and describing the differences between geometric objects, such as point clouds, geometric curves, or unparametrized surfaces~\cite{younes2010shapes,bauer2014overview,srivastava2016functional,dryden2016statistical}. The main sources of difficulty in this area are the high (infinite) dimensionality and the non-linearity of such spaces; e.g., the shape space of surfaces is an (infinite dimensional) function space modulo several finite and infinite dimensional group actions. Consequently, even simple operations such as addition or averaging are not well-defined on such spaces. Riemannian geometry has been proven to provide a successful framework to tackle this challenging task: in a Riemannian viewpoint, one considers the space of all shapes of interest (geometric objects) as an infinite dimensional manifold and equips it with an (infinite dimensional) Riemannian metric, thereby encoding the invariances of the objects in the geometry and building a convenient setup for subsequent statistical analysis. In the context of geometric curves or surfaces this approach is often referred to as Elastic Shape Analysis (ESA)~\cite{srivastava2016functional,younes1998computable}, despite the analogy to elastic stretching and bending energies being only loose~\cite{rumpf2014geometry, rumpf2015variational}.

In this article we will focus on elastic shape analysis of surfaces, i.e., we consider Riemannian metrics on the quotient space $\Imm(M,\mathbb R^3)/\Diff(M)$ of immersions modulo reparametrizations, where $M$ is a compact two dimensional manifold (the parameter space), $\Imm(M,\mathbb R^3)$ denotes the space of immersions of $M$ into $\mathbb R^3$ and $\Diff(M)$ is the diffeomorphism group of the parameter space. One can define a Riemannian metric on the quotient spaces, by considering a reparametrization invariant metric on the space of immersions, such that the projection is a Riemannian submersion. Over the past years there has been a significant amount of work dedicated to studying the mathematical properties of such metrics and in particular sufficient conditions to guarantee non-degeneracy of the geodesic distance~\cite{michor2005vanishing} and local wellposedness~\cite{bauer2011sobolev} of the geodesic equations have been derived. However, the analogs of the global existence and completeness results, that have been derived in the case of planar curves~\cite{bruveris2014geodesic,bruveris2016optimal,lahiri2015precise} are still missing. From an application point of view, the most important task is a fast and robust implementation of the geodesic boundary value problem, which in the setup of geometric statistics serves as the basis for any subsequent statistical analysis~\cite{pennec2006intrinsic}.
In general, the absence of explicit formulas for geodesics makes this a highly non-trivial task. Motivated by similar results in the case of geometric curves, several simplifying transformations have been proposed that locally flatten the Riemannian metric~\cite{kurtek2010novel,kurtek2011elastic,jermyn2017elastic}. The most successful among these is the so-called Square Root Normal Field (SRNF) transformation~\cite{jermyn2017elastic}, which assigns an invariant (pseudo) distance to the space of immersions by considering the $L^2$-distance between appropriately weighted normal vector fields. The corresponding shape distance can then be calculated by minimizing over the action of the reparametrization group, see Section~\ref{subsec:SRNF} for an exact definition of this framework. Based on the resulting computational ease and convincing results~\cite{laga2017numerical,bauer2020numerical}, the SRNF framework has been proven successful in a variety of applications, see e.g.~\cite{kurtek2014statistical,joshi2016surface,matuk2020biomedical,laga20214d}. In a recent paper~\cite{klassen2020closed} certain degeneracy results for the resulting distance have been characterized, but a more detailed theoretical study of its properties is still missing; the second part of this article will contribute towards this aim. 

The optimal mass transport (OMT) problem was first formulated as a non-convex optimization problem on the space of transport maps by Monge in 1781~\cite{monge1781memoire}. Since then a large amount of work has been dedicated to gaining a better theoretical understanding of this challenging model; we refer to the monographs~\cite{villani2008optimal,villani2003topics} for a detailed introduction to the field.
Over the past years OMT has proven successful in a variety of applications, ranging from Computer Vision to Image Analysis and in particular Statistics and Data Science, see e.g.~\cite{rubner2000earth,haker2004optimal,solomon2014wasserstein,solomon2015transportation,peyre2019computational} and the reference therein. Fueled by these applications efficient numerical discretizations of OMT have been developed~\cite{peyre2019computational,benamou2000computational,merigot2011multiscale,burkard2012assignment} including in particular the celebrated Sinkhorn algorithm~\cite{cuturi2013sinkhorn,sinkhorn1964relationship}, which efficiently solves an entropic regularized version of the OMT problem.  

The original formulation of OMT is rooted in the assumption that both densities have the same total mass. Motivated by applications, where this can be a limiting factor, various formulations of OMT that lift this restriction have been proposed~\cite{piccoli2014generalized,maas2015generalized,kondratyev2016new,liero2018optimal,chizat2018interpolating}; such transportation problems are also called unbalanced transport problems. In particular, a new family of metrics that interpolates between the Wasserstein and the Fisher-Rao metric has been introduced in~\cite{liero2018optimal,chizat2018interpolating,kondratyev2016new}. The theoretical properties of this model, called Wasserstein-Fisher-Rao distance (WFR) or Hellinger-Kantorovich distance, have been studied in detail in~\cite{liero2016optimal,chizat2018unbalanced,liero2018optimal} and, as with traditional optimal transport, efficient Sinkhorn-type, entropy regularized methods have been introduced in~\cite{chizat2018scaling}.

\subsection*{Contributions:}
We start our presentation by 
reviewing the Kantorovich formulation of the WFR distance, where we will focus on the induced distance on the subspace of all finitely supported measures on $S^2$. In this setting the computation of the WFR distance reduces to a convex optimization problem on the space of discrete semi-couplings, i.e., on the space of pairs of constrained matrices, see~Section~\ref{sec:finmeasures} and Lemma~\ref{ATilde}. In Section~\ref{flatness}, we prove that the set of measures on $S^2$ with a fixed number of support points is a locally flat metric space with respect to the WFR metric. This result is an extension of a result in \cite{chizat2018interpolating}, in which the same result is proved for measures on convex Euclidean domains instead of $S^2$. In Section~\ref{algorithm} we use the formulations of Section~\ref{sec:finmeasures} to develop an efficient coordinate descent algorithm, whose convergence is ensured by the convexity of the problem. We can find an explicit solution to the optimization problem when restricted to the space of semi-couplings with a fixed first (second, resp.) matrix, cf. Lemma~\ref{decreasing}. This gives rise to a simple, numerically efficient algorithm to compute the optimal semi-coupling. In Section~\ref{implementation} we then present an open source pytorch\footnote{Our code is available at~\url{https://github.com/emmanuel-hartman/WassersteinFisherRaoDistance}.} implementation of this algorithm and compare it in several experiments to the entropic regularized Sinkhorn solver of~\cite{chizat2018scaling}. 
 
In the second part of our article, we focus on the SRNF distance. We start by presenting a unified framework for the SRNF distance that allows us to incorporate both smooth and piecewise linear surfaces (simplicial complexes). We then show that this extended framework coincides with the original SRNF distance when restricted to smooth surfaces, cf.~Theorem~\ref{thm:PLequiv}.  
The main contribution of the second part, which establishes a connection between the SRNF distance on the space of unparametrized surfaces and the Wasserstein-Fisher-Rao distance on the space of Borel measures on $S^2$, is presented next; more precisely, we construct a map from the space of piecewise linear surfaces to the space of finitely supported Borel measures on $S^2$, such that the SRNF shape distance is the pullback of the Wasserstein-Fisher-Rao distance via this map, see~Theorem~\ref{thm:maintheorem}. 
The central building block of this result is related to the theory of area measures~\cite{alexandrov1938theorie}, which have a long history in convex geometry and in particular in Brunn-Minkowski theory, cf.~\cite{schneider2014convex} and the references therein. In the context of shape analysis of curves, area measures have been recently studied in~\cite{charon2020length}.

Theorem~\ref{thm:maintheorem} highlights the degeneracy of the SRNF distance: in the recent paper~\cite{klassen2020closed} it has been shown that there exist families of non-equivalent closed surfaces that are indistinguishable by the SRNF distance. Our result shows that for any closed surface, there exists a unique convex surface such that the SRNF distance cannot distinguish them. It turns out that this convex surface is exactly the solution of the well-known Minkowski problem~\cite{schneider1993convex,sellaroli2017algorithm}, which allows us to use algorithms of convex geometry to present 
examples of such pairs of surfaces that are indistinguishable by the SRNF, cf.~Fig.~\ref{fig:ex_conv}.

As a second outcome of Theorem~\ref{thm:maintheorem}, we obtain a new algorithm for the precise SRNF distance computation by reducing it to the solution of the Wasserstein-Fisher-Rao distance. Thus the algorithm developed in the first part of the paper directly applies to this situation. In the final section of this article, we present numerical experiments and, in particular, a comparison to previous implementations of the SRNF distance.

\subsection*{Acknowledgements.}
We thank Nicolas Charon, Ian Jermyn,  Cy Maor, Zhe Su, François-Xavier Vialard, and the statistical shape analysis group at FSU for helpful discussions during the preparation of this manuscript. 

\addtocontents{toc}{\protect\setcounter{tocdepth}{2}}
\section{Unbalanced optimal transport and the Wasserstein-Fisher-Rao distance}\label{SRMM}
In this section, we will first recall the Kantorovich formulation of the recently proposed Wasserstein-Fisher-Rao distance. We will then discuss the restriction of this distance to the space of finitely supported measures on $S^2$. In our main result of this section, we will construct an efficient splitting algorithm for the computation of this distance. We will prove the convergence of our algorithm using the result that the computation of this distance can be reduced to optimizing a concave function over a finite-dimensional convex set.

\subsection{The Wasserstein-Fisher-Rao distance}\label{SRMMDef}
In recent years, there has been a concerted effort by the optimal transport community to extend the definition of well-studied classical optimal transport to unbalanced problems, i.e. to transport problems that allow for expansion and compression of mass.
We will consider a specific example of such a generalization called the Wasserstein-Fisher-Rao distance that was introduced independently by~\cite{chizat2018interpolating} and~\cite{liero2018optimal}. In the following, we will discuss the corresponding Kantorovich formulation, as introduced in~\cite{chizat2018unbalanced}, for the special case of measures on $S^2$.

Therefore we denote by $\mathcal{M}_+(S^2)$ the space of finite Borel measures on $S^2$. To formulate the Kantorovich problem for unbalanced transport we introduce the notion of a semi-coupling, which is a direct generalization of the notion of a  coupling, which is used in standard OMT:
\begin{definition}[Semi-couplings~\cite{chizat2018unbalanced}]
Given $\mu,\nu\in\mathcal{M}_+(S^2)$ the set of all \textit{semi-couplings} is given by
\begin{equation}
    \Gamma(\mu,\nu)=\left\{ (\gamma_0,\gamma_1)\in \mathcal{M}_+(S^2\times S^2)^2| (\operatorname{Proj}_0)_\#\gamma_0=\mu,(\operatorname{Proj}_1)_\#\gamma_1=\nu\right\}.
\end{equation}
\end{definition}  
\noindent The Wasserstein-Fisher-Rao distance from $\mu$ to $\nu$ can be defined as the infimum of a functional on the space of semi-couplings of $\mu$ and $\nu$.
\begin{definition}[Wasserstein-Fisher-Rao Distance ~\cite{chizat2018unbalanced,liero2018optimal}]
The Wasserstein-Fisher-Rao Distance on $\mathcal{M}_+(S^2)$ is given by \begin{align}
    &\operatorname{WFR}:\mathcal{M}_+(S^2)\times\mathcal{M}_+(S^2)\to\R^+ \text{ defined via }\\
    &(\mu,\nu)\mapsto \inf\limits_{(\gamma_0,\gamma_1)\in \Gamma(\mu,\nu)}\left(\int_{S^2\times S^2}\left|\sqrt{\frac{\gamma_0}{\gamma}(u,v)}u-\sqrt{\frac{\gamma_1}{\gamma}(u,v)}v\right|^2 d\gamma(u,v)\right)^{1/2}
\end{align}
where $\gamma\in \mathcal{M}_+(S^2\times S^2)$ such that $\gamma\ll\gamma_0,\gamma_1$. 
\end{definition}
The following theorem summarizes the main result about this distance function:
\begin{theorem}[Properties of the Wasserstein-Fisher-Rao Distance~\cite{chizat2018interpolating,chizat2018unbalanced,liero2018optimal}]
Given any $\mu,\nu\in\mathcal{M}_+(S^2)$ there exists an optimal $(\gamma_0,\gamma_1)\in \Gamma(\mu,\nu)$ such that \begin{equation}
   \operatorname{WFR}(\mu,\nu)^2=\int_{S^2\times S^2}\left|\sqrt{\frac{\gamma_0}{\gamma}(u,v)}u-\sqrt{\frac{\gamma_1}{\gamma}(u,v)}v\right|^2 d\gamma(u,v).
\end{equation} Further,  $(\mathcal{M}_+(S^2),\operatorname{WFR})$ is a geodesic length space. 
\end{theorem}

\subsection{The subspace of all finitely supported measures}\label{sec:finmeasures}
We will now consider the restriction of this metric to the subset of all finitely supported, finite measures on $S^2$
\begin{equation}
 \mathcal{M}^{<\infty}_+(S^2)=\left\{\sum_{i=1}^m a_i \delta_{u_i}:  m\in \mathbb{N}, a_i\in \mathbb{R}^+, \text{ and } u_i\in S^2\right\} \subset\mathcal{M}_+(S^2),
\end{equation}  
where $\delta_u$ is the Dirac measure at $u\in S^2$. Note that for $\mu,\nu\in\mathcal{M}^{<\infty}_+(S^2)$ an arbitrary semi-coupling $(\gamma_0,\gamma_1)\in \Gamma(\mu,\nu)$ is not required to be finitely supported. We will introduce a subset of discrete semi-couplings and show that optimizing over all valid semi couplings is equivalent to optimizing over our restricted subset.
\begin{definition}[Discrete semi-couplings]\label{def:discrete_semicoupling}
 Let $\mu,\nu\in\mathcal{M}^{<\infty}_+(S^2)$ with $\mu=\sum_{i=1}^m a_i \delta_{u_i}$ and $\nu=\sum_{j=1}^n b_j\delta_{v_j}$. A \textit{discrete semi-coupling} of $\mu$ and $\nu$ is a pair of $(m+1)\times(n+1)$ matrices $(A,B)$ satisfying the properties:
\begin{enumerate}[{\textbf(a)}]
\item for all $i,j$, $A_{ij}\geq 0$ and $B_{ij}\geq0$;\label{property_a}
\item for all $i=1,\dots, m$, $a_i=A_{i0}+\dots+A_{in}$;\label{property_b}
\item for each  $j=0,\dots,n$, $A_{0j}=0$;\label{property_c}
\item for all $j=1,\dots,n$, $b_j=B_{0j}+\dots+B_{mj}$;\label{property_d}
\item for each $i=0,\dots, m$, $B_{i0}=0$.\label{property_e}
\end{enumerate}
We denote the set of all discrete semi-couplings of $\mu$ and $\nu$ by $\mathcal{A}(\mu,\nu)$.
\end{definition}
Note that the discrete semi-couplings from $\mu$ to $\nu$ represent a proper subset of $\Gamma(\mu,\nu)$. Thus, to show that we can compute $\operatorname{WFR}(\mu,\nu)$ by simply optimizing over $\mathcal{A}(\mu,\nu)$ we need the following lemma.
\begin{lemma}
 Let $\mu,\nu\in\mathcal{M}^{<\infty}_+(S^2)$ with  $\mu=\sum_{i=1}^m a_i \delta_{u_i}$ and $\nu=\sum_{j=1}^n b_j\delta_{v_j}$. Let $(\gamma_0,\gamma_1)\in\Gamma(\mu,\nu)$; then there exists $(A,B)\in \mathcal{A}(\mu,\nu)$ such that
\begin{equation}
      \int_{S^2\times S^2}\left|\sqrt{\frac{\gamma_0}{\gamma}(u,v)}u-\sqrt{\frac{\gamma_1}{\gamma}(u,v)}v\right|^2 d\gamma(u,v)=\sum_{i=0}^{m}\sum_{j=0}^n| \sqrt{A_{ij}}u_i-\sqrt{B_{ij}}v_j|^2
\end{equation}
where $u_0=v_0=(1,0,0)\in S^2$.
\end{lemma}
\begin{proof}
Construct $(A,B)$ as follows: 
\begin{itemize}
    \item For $i\in\{1,...,m\}$ and $j\in \{1,...,n\}$,  $A_{ij}=\gamma_0/\gamma(u_i,v_j)$ and $B_{ij}=\gamma_1/\gamma(u_i,v_j)$,
    \item for $i\in\{1,...,m\}$,  $A_{i0}=a_i-\sum_{j=1}^n\gamma_0/\gamma(u_i,v_j)$,
    \item for $j\in \{1,...,n\}$,  $B_{0j}=b_j-\sum_{i=1}^m\gamma_1/\gamma(u_i,v_j),$
    \item for $i\in\{0,...,m\}$,  $B_{i0}=0,$ and
    \item for $j\in \{0,...,n\}$,  $A_{0j}=0.$
\end{itemize}
Note that by construction $(A,B)\in \mathcal{A}(\mu,\nu)$. Observe,
\begin{align}
    \int\limits_{S^2\times S^2}\left|\sqrt{\frac{\gamma_0}{\gamma}}u-\sqrt{\frac{\gamma_1}{\gamma}}v\right|^2d\gamma(u,v)&=\int\limits_{S^2\times S^2}\left|\sqrt{\frac{\gamma_0}{\gamma}}\right|^2d\gamma(u,v)+\int\limits_{S^2\times S^2}\left|\sqrt{\frac{\gamma_1}{\gamma}}\right|^2d\gamma(u,v)\\&\qquad-2\int\limits_{S^2\times S^2}\sqrt{\frac{\gamma_0}{\gamma}}\sqrt{\frac{\gamma_1}{\gamma}}(u\cdot v)\gamma(u,v)\\
    &=\sum_{i=1}^{m}a_i+\sum_{j=1}^{n}b_j-2\sum_{i=1}^{m}\sum_{j=1}^{n}\sqrt{A_{ij}}\sqrt{B_{ij}}(u_i\cdot v_j)\\
    &=\sum_{i=0}^{m}\sum_{j=0}^n| \sqrt{A_{ij}}u_i-\sqrt{B_{ij}}v_j|^2
\end{align}
\end{proof}
\begin{corollary}
 Let $\mu,\nu\in\mathcal{M}^{<\infty}_+(S^2)$ with $\mu=\sum_{i=1}^m a_i \delta_{u_i}$ and $\nu=\sum_{j=1}^n b_j\delta_{v_j}$. Then \begin{equation}
     \operatorname{WFR}(\mu,\nu)^2=\inf\limits_{(A,B)\in\mathcal{A}(\mu,\nu)}\sum_{i=0}^{m}\sum_{j=0}^n| \sqrt{A_{ij}}u_i-\sqrt{B_{ij}}v_j|^2
 \end{equation}
\end{corollary}
\begin{remark}\label{reformulateSRMM}
By writing out the norms inside the summation and excluding terms that sum to zero, one obtains the following alternative formula for WFR:
\begin{equation}\label{eq:distance}
   \operatorname{WFR}(\mu,\nu)=\left(\sum_{i=1}^{m} a_i+\sum_{j=1}^{n} b_j-2\sup_{(A,B)\in\mathcal{A}}\sum_{i=1}^{m}\sum_{j=1}^{n}\sqrt{A_{ij}B_{ij}}(u_i\cdot v_j)     \right)^{1/2}.
\end{equation}
\end{remark}

\begin{remark}Given a discrete semi-coupling $(A,B)$, the zeroth column of $A$ and zeroth row of $B$ are included to handle the case where \textit{all} of the mass at the corresponding support is destroyed/created rather than transported. This does not mean, however, that these rows/columns being zero correspond to no creation/destruction of mass. We will demonstrate this fact in the following example. \end{remark}

\begin{example}
We consider the example with $\mu=2\delta_{u_1}+\delta_{u_2}$ and $\nu=\delta_{v_1}+3\delta_{v_2}$ 
with 
$$u_1=v_1=\left(
\begin{array}{c}
1\\
0\\
\end{array}
\right),\qquad u_2=v_2=\left(
\begin{array}{c}
0\\
1\\
\end{array}
\right).$$
The corresponding weight matrix between the supports is then given by 
\begin{equation}
    \Omega=\begin{pmatrix}1&0\\0&1\end{pmatrix}.
\end{equation} 
Consequently the optimal semi-coupling is given by:
\begin{equation}
    A=\begin{pmatrix}0&0&0\\0&2&0\\0&0&1\end{pmatrix} \text{ and } B=\begin{pmatrix}0&0&0\\0&1&0\\0&0&3\end{pmatrix}.
\end{equation}
Notice that even though the zeroth column of $A$ and zeroth row of $B$ are all zeros, a unit of mass is destroyed and two units of mass are created. 
For contrast we consider a second example with the same masses but different supports: 
$\mu=2\delta_{u_1}+\delta_{u_2}$ and $\nu=\delta_{v_1}+3\delta_{v_2}$ where 
$$u_1=v_1=\left(
\begin{array}{c}
1\\
0\\
\end{array}
\right),\qquad u_2=\left(
\begin{array}{c}
0\\
1\\
\end{array}
\right),\qquad v_2=-u_2=-\left(
\begin{array}{c}
0\\
1\\
\end{array}
\right).$$
This time the corresponding weight matrix between the supports is given by\begin{equation}
    \Omega=\begin{pmatrix}1&0\\0&-1\end{pmatrix}
\end{equation} 
and an optimal semi-coupling is given by:
\begin{equation}
    A=\begin{pmatrix}0&0&0\\0&2&0\\1&0&0\end{pmatrix} \text{ and } B=\begin{pmatrix}0&0&3\\0&1&0\\0&0&0\end{pmatrix}.
\end{equation}
In this case, two units of mass are destroyed (one from each support of $\mu$) and three units are created (all three are created at the second support of $\nu$).
\end{example}

\subsection{Flatness of the space of measures with a fixed number of support points}\label{flatness}
Note that while the main result of this section proves a fundamental fact about the geometry of the space of finitely supported measures, this result is not a prerequisite for the rest of the paper. Let ${\mathcal M}^n_+$ denote the set of measures on $S^2$ that are supported at precisely $n$ points. It is clear that ${\mathcal M}^n_+$ has a natural structure as a $3n$-manifold. In the following theorem, we will show that the restriction of the WFR metric to ${\mathcal M}^n_+$ is locally isometric to the distance function corresponding to a flat Riemannian metric on this space.

\begin{theorem} \label{FlatSubspace}There is a flat Riemannian metric on  ${\mathcal M}^n_+$ whose Riemannian distance function agrees with the Wasserstein-Fisher-Rao metric on a small neighborhood of every point. 

Furthermore, let 
$${\mathcal P}_n=\left\{(x_1,\dots,x_n)\in \left(\reals^3-\{0\}\right)^n: \text{ for all }i\neq j, 
\frac{x_i}{|x_i|}\neq\frac{x_j}{|x_j|}\right\}\subset \left(\reals^3/\{0\}\right)^n$$
and define
\begin{equation}\label{eq:Qn}
   Q_n:{\mathcal P}_n\to {\mathcal M}^n_+\text{ by } Q_n(x_1,\dots,x_n)=\sum_{i=1}^n\|x_i\|^2
\delta_{\left(x_i/\|x_i\|\right)}. 
\end{equation}
Then, for every $x=(x_1,\dots,x_n)\in{\mathcal P}_n$ there exists an $\epsilon>0$ such that for all $y,z\in{\mathcal P}_n$ with
 $|y-x|<\epsilon$ and $|z-x|<\epsilon$ the optimal discrete semi-coupling from $Q_n(y)$ to $Q_n(z)$ is a pair of diagonal matrices.
\end{theorem}

\begin{remark}Theorem~\ref{FlatSubspace} is closely related to Theorems 4.1 and 4.2 of \cite{chizat2018interpolating}. The main difference is that in \cite{chizat2018interpolating}, the domain of the measure is assumed to be a convex region in $\reals^d$, whereas in Theorem \ref{FlatSubspace}, the domain is $S^2$. Another difference is that in \cite{chizat2018interpolating}, a lower bound on the size of the neighborhood on which the metric is flat is given. The proof given here is more elementary (a straightforward application of differential topology) and thus we hope it is of interest in itself. 
\end{remark}
Before we are able to prove Theorem~\ref{FlatSubspace}, we need the following technical lemma.
\begin{lemma}\label{MinHess}
Let $M$ be an $m$-dimensional manifold, and $C$ a compact subset of $M$. Let $f:\reals^n\times M\to\reals$ be a $C^\infty$ function; let $x_0\in C\subset M$. Define $g:M\to\reals$ by $g(x)= f(0,x)$. Assume that $g$ restricted to $C$ attains an absolute minimum at $x_0$, and that $x_0$ is the only point of $C$ where $g$ achieves this minimum. Also, assume that $g$ has a critical point at $x_0$ and that the Hessian of $g$ at $x_0$ is positive definite. (Clearly this assumption is independent of which chart containing $x_0$ is used to compute the Hessian.) Now, suppose that there exists $\zeta>0$ such that for all $p\in\reals^n$ with $|p|<\zeta$, the function $M\to \reals$ defined by $x\mapsto f(p,x)$ has a critical point at $x_0$. Then there exists $\epsilon>0$ such that for all $p\in\reals^n$ with $|p|<\epsilon$, the function $C\to\reals$ defined by $x\mapsto f(p,x)$ attains an absolute minimum at $x_0$. 
\end{lemma}

\begin{proof} Suppose the lemma is false. Then there exist sequences $\{p_k\}\in\mathbb{R}^n$ with $p_k\to 0$ and $\{x_k\}\in C$ such that for all $k\in\mathbb{N}$, $f(p_k,x_k)< f(p_k,x_0)$. By compactness of $C$, we may choose a subsequence of $\{x_k\}$ that converges in $C$. Continue to denote this subsequence by $\{x_k\}$ (and denote the corresponding subsequence of $\{p_k\}$ by $\{p_k\}$). Let $c=\lim x_k$. There are two  cases to consider:
\smallskip

\noindent{\textbf Case 1.} Suppose $c\neq x_0$. In that case, by the continuity of $f$, $f(0,c)=\lim f(p_k,x_k)\leq \lim f(p_k,x_0)=f(0,x_0)$, contradicting the hypothesis $g$ achieves its minimum on $C$ only at the point $x_0$. 

\smallskip

\noindent{\textbf Case 2.} Suppose $c=x_0$. By taking subsequences, we may assume that all of the $x_k$ lie in a single chart of $M$; hence, for the rest of this proof we will replace $M$ by $U\subset\reals^m$ and $x_0$ by $0\in \reals^m$. Recall that one assumption in our Lemma is that the Hessian of $g$ at $0$ is positive definite; denote its smallest eigenvalue by $\lambda$. (Using the chart, we are now thinking of $g$ as a function $U\to\reals$.) Clearly the condition that the lowest eigenvalue of the Hessian of the map $y\mapsto f(p,y)$ at $x$ is greater than $\lambda/2$ is an open condition on $(p,x)$. Hence we can choose $\delta>0$ so that for all $|p|<\delta$ and $|x|<\delta$, this condition is satisfied. Choose $k_0$ such that $|p_{k_0}|<\delta$ and $|x_{k_0}|<\delta$. Define the map $\tilde g: N_\delta(0)\to\reals$ by $\tilde g(x)=f(p_{k_0},x)$. Define $L:[0,1]\to \reals$ by $L(t)=f(p_{k_0},tx_{k_0})$. Using the chain rule, we see that $L''(t)=x_{k_0}^TH(p_{k_0},tx_{k_0})x_{k_0}$, where $H$ denotes the Hessian of $\tilde g$. Since we are assuming that the smallest eigenvalue of $H$ on $N_\delta(0)\times N_\delta(0)$ is greater than $\lambda/2$, it follows that $L''(t)>|x_{k_0}|^2\lambda/2$ for all $t\in [0,1]$. Also, recall the assumption that $0$ is a critical point of $\tilde g$; hence $L'(0)=0$. From this, it follows that $f(p_{k_0},x_{k_0})=L(1)>L(0)=f(p_{k_0},0)$. This contradicts our original construction of the sequences $\{p_k\}$ and $\{x_k\}$, which required that for all $k$, $f(p_k,x_k)<f(p_k,0)$, thereby proving the lemma.
\end{proof}

We are now able to proceed with the proof of Theorem~\ref{FlatSubspace}.
\begin{proof}[Proof of Theorem~\ref{FlatSubspace}] 
Clearly, the set ${\mathcal P}_n$ is a $3n$-manifold, and the symmetric group $\Sigma_n$ acts freely on ${\mathcal P}_n$ in the obvious way. Furthermore, if we give ${\mathcal P}_n$ the standard Euclidean Riemannian metric restricted from $\reals^{3n}$, $\Sigma_n$ acts by isometries. It follows that ${\mathcal P}_n/\Sigma_n$ inherits the structure of a flat Riemannian manifold.

Let $\mu\in{\mathcal M}^n_+$; so we can write $\mu=\sum_{i=1}^n a_i\delta_{u_i}$, where $\{u_1,\dots, u_n\}$ is a set of distinct elements of $S^2$ and each $a_i>0$. 
 Clearly the map $Q_n$, as defined in~\eqref{eq:Qn}, induces a bijection ${\mathcal P}_n/\Sigma_n\to{\mathcal M}^n_+$.

To complete the proof of the first statement of Theorem ~\ref{FlatSubspace}, we just need to prove that  $Q_n$ maps a small neighborhood of  each point $x$ in ${\mathcal P}_n$ isometrically to a small neighborhood of $Q_n(x)$ in ${\mathcal M}^n_+$.
Indeed this will follow directly from the second statement, i.e., from the fact that if two measures are supported by the same number of points and have their support points and weights very close to each other, then the optimal discrete semi-coupling between them is given by the obvious diagonal matrices coming from the 1-1 correspondence between their supports and weights. 

Let $\mu$ and $\nu$ be two measures on $S^2$, both supported at precisely $n$ points. Write $\mu=\sum_{i=1}^n a_i \delta_{u_i}$ and $\nu=\sum_{i=1}^n b_i \delta_{v_i}$, where $u_1,\dots,u_n$ are distinct elements of $S^2$ and $a_1,\dots,a_n>0$; similarly, $v_1,\dots,v_n$ are distinct elements of $S^2$ and $b_1,\dots,b_n>0$. Recall that a {\em discrete semi-coupling} from $\mu$ to 
$\nu$ is defined to be a pair of $(n+1)\times (n+1)$ matrices $(A,B)$ (where both indices in each matrix run from $0$ to $n$)  satisfying the following conditions:
\begin{enumerate}
\item for all $i,j$, $A_{ij}\geq 0$ and $B_{ij}\geq0$;
\item for all $i=1,\dots, n$, $a_i=A_{i0}+\dots+A_{in}$;
\item for all  $j=0,\dots,n$, $A_{0j}=0$;
\item for all $j=1,\dots,n$, $b_j=B_{0j}+\dots+B_{nj}$;
\item for all $i=0,\dots, n$, $B_{i0}=0$.
\end{enumerate}
We define a cost function $C$ on the set of all discrete semi-couplings from $\mu$ to $\nu$ by
$$C(A,B)=\sum_{i=0}^{m}\sum_{j=0}^{n}\left|\sqrt{B_{ij}}v_j-\sqrt{A_{ij}}u_i\right|^2$$
An {\em optimal discrete semi-coupling} from $\mu$ to $\nu$ is defined to be a discrete semi-coupling that minimizes $C$. Note that the set of discrete semi-couplings (i.e., the domain of $C$) varies according to the particular measures $\mu$ and $\nu$. To remedy this inconvenient fact, we define a {\em normalized discrete semi-coupling} to be pair of matrices $(A,B)$ satisfying the conditions:
\begin{enumerate}
\item for all $i,j$, $A_{ij}\geq 0$ and $B_{ij}\geq0$;
\item for all $i=1,\dots, n$, $1=A_{i0}+\dots+A_{in}$;
\item for each  $j=0,\dots,n$, $A_{0j}=0$;
\item for all $j=1,\dots,n$, $1=B_{0j}+\dots+B_{nj}$;
\item for each $i=0,\dots, n$, $B_{i0}=0$.
\end{enumerate}
We then define the normalized cost function by
$$\bar C(A,B)=\sum_{i=0}^{n}\sum_{j=0}^{n}\left|\sqrt{b_jB_{ij}}v_j-\sqrt{a_iA_{ij}}u_i\right|^2,$$ 
where $a_0$ and $b_0$ should be assigned the value 0.

Note that we are just rescaling each row of $A$ and each column of $B$ to make their sums 1, and then inserting the relevant scalars into the cost function so as not to change its behavior.
Define an {\em optimal normalized discrete semi-coupling} from $\mu$ to $\nu$ to be a normalized discrete semi-coupling that minimizes the normalized cost function. Note that the set of normalized discrete semi-couplings no longer depends on the particular pair of measures $\mu$ and $\nu$ (as long as they are both supported at precisely $n$ points). However, the normalized cost function does depend on $\mu$ and $\nu$. Henceforth, to make this dependence explicit, we write $\bar C_{\mu,\nu}$ to denote the normalized cost function corresponding to the measures $\mu$ and $\nu$.

We now make a further change of variables. For each normalized discrete semi-coupling $(A,B)$, we define another pair of $(n+1)\times(n+1)$ matrices $(P,Q)$ as follows: For each $i,j$, $P_{i,j}=\sqrt{A_{i,j}}$ and $Q_{i,j}=\sqrt{B_{i,j}}$. If we write the cost function as a function of $(P,Q)$, the form of the cost function becomes

$$G_{\mu,\nu}(P,Q)=\sum_{i=0}^{n}\sum_{j=0}^{n}\left|Q_{i,j}\sqrt{b_j}v_j-P_{i,j}\sqrt{a_i}u_i\right|^2.$$

Note that the corresponding domain of $G_{\mu,\nu}$ is the set of all pairs $(P,Q)$ of $(n+1)\times(n+1)$ matrices satisfying the following:

\begin{enumerate}
\item for all $i,j$, $P_{ij}\geq 0$ and $Q_{ij}\geq0$;
\item for all $i=1,\dots, n$, $1=P_{i0}^2+\dots+P_{in}^2$;
\item for each  $j=0,\dots,n$, $P_{0j}=0$;
\item for all $j=1,\dots,n$, $1=Q_{0j}^2+\dots+Q_{nj}^2$;
\item for each $i=0,\dots, n$, $Q_{i0}=0$.
\end{enumerate}
Let ${\mathcal T_0}$ denote the domain of $G_{\mu,\nu}$; it is a product of positive orthants of unit spheres, since each of rows 1 through $n$ of $P$, and each of columns 1 through $n$ of Q is constrained to be in such an orthant. (Row 0 of $P$ and column 0 of $Q$ are each required to be the zero vector.) Also, note that we can remove the first condition on the domain of $G_{\mu,\nu}$, thereby extending its domain to be a $2n$-fold product of spheres, with each sphere having dimension $n$.  Denote this extended domain by ${\mathcal T}$. Clearly, $G_{\mu,\nu}$ is defined  (by the same formula) and smooth on ${\mathcal T}$. 

Now consider the case $\mu=\nu$. In this case,  $G_{\mu,\mu}$ achieves a minimum value of $0$ on ${\mathcal T_0}$, and this value is achieved only at the point where
$$P=Q=
\begin{pmatrix}
0&0\cr
0&I_n
\end{pmatrix}.
$$
Denote this particular discrete semi-coupling by $(P_0,Q_0)$.
In fact, it's clear that this value is also a minimum on all of ${\mathcal T}$ since $G_{\mu,\mu}$ is a sum of non-negative terms. (On all of ${\mathcal T}$ there are other points besides $(P_0, Q_0)$ where this minimum is achieved.) 
$(P_0, Q_0)$ is a critical point of $G_{\mu,\mu}$ on ${\mathcal T}$, since it is a point where the minimum is achieved. Furthermore, we will show that the Hessian of $G_{\mu,\mu}$ at $(P_0, Q_0)$ is positive definite. 

To see that $G_{\mu,\mu}$ has positive definite Hessian at $(P_0,Q_0)$, reason as follows: Let $(P(t),Q(t))$, for $-a<t<a$, be a path in ${\mathcal T}$ such that $(P(0),Q(0))=(P_0,Q_0)$, and $(P'(0),Q'(0))\neq (0,0)$.  We will now show that $G_{\mu,\mu}(P(t),Q(t))$ has positive second derivative at $t=0$, no matter which such path is chosen. Note that no summand $G_{\mu,\mu}(P(t),Q(t))$ can have negative second derivative, since $(P(0),Q(0))$ is a local minimum of each summand. Hence, it suffices to show that just one summand of $G_{\mu,\mu}(P(t),Q(t))$ has positive second derivative. Note that the diagonal entries of $P'(0)$ and $Q'(0)$ are all zero, since ${\mathcal T}$ is a product of spheres.  It follows that either $P'(0)$ or $Q'(0)$ must have a nonzero entry in an off-diagonal element. Thus, choose $i_0\neq j_0$, and assume that $P'_{i_0,j_0}(0)\neq 0$. (The reasoning is the same for the case $Q'_{i_0,j_0}(0)\neq 0$.) First consider the case $j_0>0$. The corresponding summand of $G_{\mu,\mu}(P(t),Q(t))$ is $$\left|Q_{i_0,j_0}(t)\sqrt{b_{j_0}}v_{j_0}-P_{i_0,j_0}(t)\sqrt{a_{i_0}}u_{i_0}\right|^2.$$ A straightforward computation shows that the second derivative of this summand at $t=0$ is $$2\left|Q'_{i_0,j_0}(0)\sqrt{b_{j_0}}v_{j_0}-P'_{i_0,j_0}(0)\sqrt{a_{i_0}}u_{i_0}\right|^2.$$ Since $i_0\neq j_0$, it follows that $u_{i_0}$ and $v_{j_0}$ are linearly independent. Hence this second derivative is positive, since we are assuming that $P'_{i_0,j_0}(0)\neq 0$. For the case $j_0=0$, we know that $Q_{i_0,j_0}(t)$ is constant at zero (since it stays in ${\mathcal T}$). Therefore $Q'_{i_0,j_0}(0)=0$, and it still follows that the second derivative of this summand is positive. This proves that $G_{\mu,\mu}(P(t),Q(t))$ has positive second derivative at $t=0$ for every path $(P(t),Q(t))$ in ${\mathcal T}$ with $(P(0),Q(0))=(P_0,Q_0)$. Hence the Hessian of $G_{\mu,\mu}$ at $(P_0,Q_0)$ is positive definite.

\noindent{\textbf Claim:} For all $\mu,\nu\in{\mathcal M}_n^0$, $(P_0,Q_0)$ is a critical point of $G_{\mu,\nu}$ on ${\mathcal T}$.

The proof of this Claim is a straightforward computation. First we construct a basis for the tangent space $T_{(P_0,Q_0)}{\mathcal T}$, using the fact that ${\mathcal T}$ is a product of unit spheres. This basis is a union of two sets. The first set consists of tangent vectors of the form $(E_{i,j},0)$, where $E_{i,j}$ has the entry 1 in the $(i,j)$-th place and 0's elsewhere, for $1\leq i\leq n$, $0\leq j\leq n$, and $i\neq j$. The second set consists of elements of the form $(0,E_{i,j})$, where $E_{i,j}$ has the entry $1$ in the $(i,j)$-th space, and 0's elsewhere, for $0\leq i\leq n$, $1\leq j\leq n$, and $i\neq j$. It is then a trivial  calculation to see that the directional derivative of $G_{\mu,\nu}$ at $(P_0,Q_0)$ in the direction of any of these tangent vectors vanishes. This proves the claim.

Let $\mu_0\in{\mathcal M}_n^0$. We need to prove that there exists an $\epsilon>0$ such that for all $(\mu,\nu)$ in an $\epsilon$-neighborhood of $(\mu_0,\mu_0)$ in ${\mathcal M}_n^0\times{\mathcal M}_n^0$, the minimum value of $G_{\mu,\nu}$ on ${\mathcal T}_0$ is achieved at $(P_0,Q_0)$.   

Define a function $H:({\mathcal P}^n\times{\mathcal P}^n)\times{\mathcal T}\to\reals$ by
$$H((x,y),(P,Q))=G_{Q_n(x),Q_n(y)}(P,Q).$$ Clearly, $H$ is smooth. We are now in a position to apply Lemma ~\ref{MinHess}. The compact set ${\mathcal T}_0$ plays the role of $C$ in the Lemma. We have shown that $(P_0,Q_0)$ is a critical point of $(P,Q)\mapsto H((x,y),(P,Q))$, for all $x,y\in{\mathcal P}_n$, and that  for any $x_0\in{\mathcal P}_n$, $(P_0,Q_0)$ is the location of the unique minimum of $(P,Q)\mapsto  H((x_0,x_0),(P,Q))$ on ${\mathcal T_0}$. Having also verified the condition on the Hessian, we can conclude from Lemma ~\ref{MinHess} that there exists an $\epsilon>0$ such that for all $(x,y)$ with $|x-x_0|<\epsilon$ and $|y-x_0|<\epsilon$, $H((x,y),(\cdot,\cdot))$ attains its maximum (on the domain ${\mathcal T}_0$) at the point $(P_0,Q_0)$. Thus the proof of the second statement follows, which implies at the same time the first statement.
\end{proof}

\subsection{A convex optimization problem}\label{algorithm}
In this section we will study the class of optimization problems, that consist of maximizing the function
\begin{align}
 F:\mathcal{A}&\to \R \qquad \text{ defined via }\\ 
 (A,B)&\mapsto \sum_{i=1}^{m}\sum_{j=1}^{n}\sqrt{A_{ij}B_{ij}}\Omega_{ij}.
\end{align} 
Here $\Omega_{ij}$ is any given weight matrix. Recall that the tuples $(A,B)\in \mathcal A$ are subject to the constraints \eqref{property_a}-\eqref{property_e}.
In the main result of this part we will  explicitly construct a sequence of semi-couplings that converges to an optimizer of $F$.

Our motivation for studying this class of optimization problems stems from the fact that for a particular choice of $\Omega$ it is equivalent to calculating the WFR-distance between two finitely supported measures. This will lead directly to an efficient algorithm to numerically calculate this distance. Note, that most existing methods for estimating the WFR-distance solve an entropy regularized problem. Our proposed solution instead converges to the WFR distance by performing an optimization on the polytope of discrete semi-couplings.

To see this connection between the function $F$ and the WFR-distance let $\mu=\sum\limits_{i=1}^m a_i \delta_{u_i}$ and $\nu=\sum\limits_{i=1}^n b_i \delta_{v_i}$ be two finitely supported, finite measures.
Recall from Remark~\ref{reformulateSRMM} that  the distance from $\mu$ to $\nu$ can be written as
\begin{equation}
    \operatorname{WFR}(\mu,\nu)=\left(\sum_{i=1}^{m} a_i+\sum_{j=1}^{n} b_j-2\sup_{(A,B)\in\mathcal{A}}\sum_{i=1}^{m}\sum_{j=1}^{n}\sqrt{A_{ij}B_{ij}}(u_i\cdot v_j)     \right)^{1/2}.
\end{equation}
Thus computing $\operatorname{WFR}(\mu,\nu)$ is equivalent to finding $(A,B)\in \mathcal{A}$ that achieves the supremum
\begin{equation}
    \sup_{(A,B)\in \mathcal{A}}\sum_{i=1}^{m}\sum_{j=1}^{n}\sqrt{A_{ij}B_{ij}}(u_i\cdot v_j),
\end{equation} 
which corresponds to the function $F$ with $\Omega$ being given by $\Omega_{ij}=(u_i\cdot v_j)$. 

\begin{remark}[Generalizations of the Algorithm]
In \cite{chizat2018unbalanced}, the Wasserstein Fisher Rao distance is formulated on a domain $M$ with a parameter $\rho\in (0,\infty)$ as \begin{equation} 
\frac{1}{\rho^2}\operatorname{WFR}^2_\rho(\mu,\nu)=\mu(M)+\nu(M)-2 \sup\limits_{(\gamma_0,\gamma_1)\in \Gamma(\mu,\nu)}\int_{d(x,y)<\pi} \cos(d(x,y)/\rho)d(\sqrt{\gamma_0\gamma_1})(x,y)
\end{equation}
where $\sqrt{\gamma_0\gamma_1}=(\frac{\gamma_0}{\gamma}\frac{\gamma_1}{\gamma})^{1/2}$ for $\gamma$ such that $\gamma_0,\gamma_1\ll\gamma$. Another popular distance is the so-called Gaussian Hellinger-Kantorovich distance, which can be expressed with a parameter $\rho\in (0,\infty)$ as
\begin{equation} 
\frac{1}{\rho^2}\operatorname{GHK}^2_\rho(\mu,\nu)=\mu(M)+\nu(M)-2 \sup\limits_{(\gamma_0,\gamma_1)\in \Gamma(\mu,\nu)}\int e^{\frac{-d^2(x,y)}{2\rho}}d(\sqrt{\gamma_0\gamma_1})(x,y).
\end{equation} For finitely supported measures $\mu,\nu$ these distances can be generically expressed as
\begin{equation}
    \frac{1}{\rho^2}\operatorname{Dist}^2_\rho(\mu,\nu)=\sum_{i=1}^{m} a_i+\sum_{j=1}^{n} b_j-2\sup_{(A,B)\in\mathcal{A}}\sum_{i=1}^{m}\sum_{j=1}^{n}\sqrt{A_{ij}B_{ij}}\Omega_{ij}
\end{equation} where
\begin{equation}
    \operatorname{Dist}_\rho(\mu,\nu)=\begin{cases} \operatorname{WFR}_\rho(\mu,\nu)&\text{ if }\Omega_{ij}=\cos(d(u_i,v_j)/\rho)\\\operatorname{GHK}_\rho(\mu,\nu)&\text{ if }\Omega_{ij}=e^{\frac{-d^2(u_i,v_j)}{2\rho}}\end{cases}.
\end{equation}\\ 
As our main result is that the SRNF can be written as the pullback of the WFR distance with $\rho=1$ for measures on $M=S^2$, we developed the algorithm with this particular case in mind. However, our algorithm does not depend on the particular choice of $\Omega$ and thus it could be also used for these more general situations.
\end{remark}

First we will show that the optimum of $F$ will be obtained on a set such that mass is never transported between supports $u_i$ and $v_j$ where $\Omega_{ij}$ is negative. Moreover, we show that $F$ obtains its optimum when all of the mass of a given support is created/destroyed if and only if none of it can be transported for a suitably small enough cost. In terms of the discrete semi-couplings this is represented by non-zero entries in the zeroth column of $A$ and the zeroth row of $B$. 
Therefore we define the subset $\tilde{\mathcal A}$ of $\mathcal{A}$ of all semi-couplings $(A,B)\in\mathcal{A}$ satisfying
\begin{enumerate}
    \item $A_{ij}=B_{ij}=0\operatorname{ whenever }\Omega_{ij}\leq 0, $
    \item $B_{0j}=0 \text{ if there exists } j \text{ such that }\Omega_{ij}> 0$, and
    \item $A_{i0}=0 \text{ if there exists } i \text{ such that }\Omega_{ij}> 0$.
\end{enumerate}
We will now show that the value function $F$ obtains its maximum on $\tilde{\mathcal A}$ and that $F$ is concave when restricted to this subset.

\begin{lemma} \label{ATilde}
The value function $F$ obtains its maximum on $\tilde{\mathcal A}$. Moreover, $\tilde{\mathcal A}$ is a convex, compact subset of $\R^{mn+m+n+1}$ and $F$ is concave when restricted to $\tilde{\mathcal{A}}$.
\end{lemma}

\begin{proof}
To show this statement, let $(A,B)\in\mathcal{A}$. We will construct a new semi-coupling $(\tilde{A},\tilde{B})\in\tilde{\mathcal A}$ such that $F(A,B)\leq F(\tilde{A},\tilde{B})$. Begin by letting $(\tilde{A},\tilde{B})=(A,B)$. Next for each $i\in\{1,...,m\}$ define 
\begin{equation}
    k_i=\begin{cases}\argmax\limits_{k\in\{1,...,n\}}\Omega_{ik}&\text{if  } \max\limits_{k\in\{1,...,n\}} \Omega_{ij}>0\\0&\text{otherwise}\end{cases}
\end{equation}
and for each $j=\{1,...,n\}$ define
\begin{equation}
    l_j=\begin{cases}\argmax\limits_{l\in\{1,...,m\}}\Omega_{lj}&\text{if  } \max\limits_{l\in\{1,...,m\}} \Omega_{lj}>0\\0&\text{otherwise}\end{cases}.
\end{equation}

\textbf{Step 1: }If $\Omega_{ij}<0$, we modify $\tilde{A}_{ij}$ and $\tilde{B}_{ij}$ as follows:
\begin{itemize}
    \item Replace $\tilde{A}_{ik_i}$ by $A_{ik_i}+A_{ij}$ and replace $\tilde{A}_{ij}$ by $0$.
    \item Replace $\tilde{B}_{l_jj}$ by $B_{l_jj}+B_{ij}$ and replace $\tilde{B}_{ij}$ by $0$.
\end{itemize}
It is clear that the new point satisfies property (1). In the next stage of our modification we will ensure that our semi-coupling also satisfies (2) and (3). 

\textbf{Step 2:}
\begin{itemize}
    \item If $k_i\neq0$, then replace $\tilde{A}_{ik_i}$ by $A_{ik_i}+A_{i0}$ and $\tilde{A}_{i0}$ by $0$. 
    \item If $l_j\neq0$, then replace $\tilde{B}_{l_jj}$ by $B_{l_jj}+B_{0j}$ and $\tilde{B}_{0j}$ by $0$.
\end{itemize}
It is clear that $(\tilde A, \tilde B)$ is an element of $\tilde{\mathcal A}$ and that $(\tilde A, \tilde B)=(A,B)$ if $(A,B)$ was already in $\tilde{\mathcal A}$.
Moreover, for each $i\in\{1,...,m\}$ and $j=\{1,...,n\}$, \begin{equation}A_{ij}B_{ij}\Omega_{ij}\leq\tilde A_{ij}\tilde B_{ij}\Omega_{ij}\end{equation} and thus $F(A,B)\leq F(\tilde A, \tilde B)$. In fact, it is clear that the value of $F$ steadily increases along the line from the original point to the new point. This contradicts the assumption that $F$ attained a local maximum at the original point. Thus $F$ obtains its maximum on the subset $\tilde{\mathcal A}$. A parametrized straight line in $\tilde{\mathcal A}$ will be of the form $\{p_{ij}t+q_{ij}, y_{ij}t+z_{ij}\}$. If we restrict $F$ to such a line, it will be a linear combination of functions of the form $\sqrt{(pt+q)(yt+z)}$, with positive coefficients. It is easy to verify that a function of this form always has a second derivative $\leq 0$ for all values of $t$ for which it is defined (i.e., for all values of $t$ for which the quantity under the square root sign is $\geq 0$). It follows that along any line in the domain, either $F$ has at most one local maximum, or it is constant along that line. Therefore, if $F$ has two local maxima on its domain, then $F$ must be constant on the entire line through these two maxima. Thus, $F$ is concave when restricted to $\tilde{\mathcal{A}}$.
\end{proof}

Thus, we have reformulated our optimization problem as finding the maximum of a concave function over a convex set (or equivalently the minimum of a convex function over a convex set). Next, we will introduce two operators on $\tilde{\mathcal A}$:
First, let 
\begin{equation}
T_1:\tilde{\mathcal A}\to \tilde{\mathcal A}\text{ is defined via }(A,B)\mapsto (E,B)\text{ where }
\end{equation}
\begin{equation}
    E_{ij}=\begin{cases}
    A_{ij} &\text{if  $i=0$ or $j=0$}\\
    \frac{a_i B_{ij}\Omega_{ij}^2}{\bsum{k=1}{n}B_{ik}\Omega_{ik}^2}&\text{if } \bsum{k=1}{n}B_{ik}\Omega_{ik}^2>0\\0& \text{otherwise}\end{cases}.
\end{equation}
Similarly let 
\begin{equation}
    T_2:\tilde{\mathcal A}\to \tilde{\mathcal A}\text{ be defined via
    }(A,B)\mapsto (A,E)\text{ where }
\end{equation}
\begin{equation}
    E_{ij}=\begin{cases}
    B_{ij} &\text{if  $i=0$ or $j=0$}\\
    \frac{b_j A_{ij}\Omega_{ij}^2}{\bsum{k=1}{m}A_{kj}\Omega_{kj}^2}&\text{if } \bsum{k=1}{m}A_{kj}\Omega_{kj}^2>0\\0&\text{otherwise }\end{cases}.
\end{equation}

We will define a sequence of semi-couplings recursively by initializing the sequence at some $(A, B)_0$
 in the interior of $ \tilde{\mathcal A}$ and iteratively updating the two components of this semi-coupling via $T_1$ and $T_2$. The primary goal of this section  will be to show that any limit point of this sequence is a maximizer of $F$, which is formalized in the following theorem:
\begin{theorem}\label{convergence}
Define a sequence of semi-couplings via  $(A,B)_{k+1}:=T_2(T_1((A,B)_k))$ where $(A,B)_0$ is in arbitrary chosen initialization from the interior of $\tilde{\mathcal A}$. Then the real-valued sequence given by $F((A,B)_{k})$ converges to the optimum of $F$, i.e.,
\begin{equation}\lim\limits_{k\to \infty}F((A,B)_{k})=\argmax\limits_{(A,B)\in\tilde{\mathcal A}}F(A,B).\end{equation} 
\end{theorem}
The key ingredient for the proof of this statement is the observation that applying $T_1$ maximizes $F$ when we fix the second component and similarly for $T_2$ when fixing the first component. This will allow us to apply results from coordinate descent analysis to obtain the proof of the above theorem. This will be the content of the following lemma:

\begin{lemma}\label{decreasing} 
The operators $T_i$ restrict to operators on the interior of $\tilde{\mathcal A}$, i.e., 
$T_i(\tilde{\mathcal A}^{\mathrm{o}})\subset \tilde{\mathcal A}^{\mathrm{o}}$. 
For a fixed $(A,B)$ in the interior $\tilde{\mathcal A}$ let $\tilde{\mathcal A}_{(-,B)}$ be the space of semi-couplings where the second factor is equal to $B$ and $\tilde{\mathcal A}_{(A,-)}$ be the space of semi-couplings where the first factor is equal to $A$. Thus, 
\begin{enumerate}
     \item $T_1(A,B)$ uniquely attains $\sup\limits_{(A^{'},B)\in \tilde{\mathcal A}_{(-,B)}}F(A^{'},B)$
    \item  $T_2(A,B)$ uniquely attains $\sup\limits_{(A,B^{'})\in \tilde{\mathcal A}_{(A,-)}}F(A,B^{'})$
\end{enumerate}
\end{lemma}

\begin{proof}

Let$(A,B)$ be in the interior $\tilde{\mathcal A}$. Thus, for any $i=1,...,m$ and $j=1,...n$ such that $\Omega_{ij}>0$ we have that $A_{ij}>0$ and $B_{ij}>0$. We need to show that for any semi-coupling such that the second matrix is equal to $B$, $F(A^{'},B)\leq F\left(T_1(A,B)\right)$ where equality holds if and only if $(A^{'},B)=T_1(A,B)$. 

Therefore, let $(A^{'},B)$ be an arbitrary semi-coupling such that the second matrix is equal to $B$ and let $(E,B)=T_1(A^{'},B)$.  
For $1\leq i \leq m$ we let
\begin{equation}
\rho_i=\sum\limits_{j=1}^{n} B_{ij}\Omega_{ij}^2\;.
\end{equation}
Consider the case where $\rho_i=0$. Since $(A,B)$ is in $\tilde{\mathcal{A}}$, $B_{ij}\geq 0$ and $B_{ij}=0$ if and only if $\Omega_{ij}^2=0$. Thus, $B_{ij}=0$ for all $1\leq j\leq n$. Thus,
\begin{align}
    \bsum{j=1}{n}\sqrt{A^{'}_{ij}B_{ij}}\Omega_{ij}=\bsum{j=1}{n}\sqrt{E_{ij}B_{ij}}\Omega_{ij}=0,
\end{align}
It remains to prove the statement for the case that $\rho_i>0$.  By Cauchy-Schwarz we have
\begin{align}
    \bsum{j=1}{n}\sqrt{A^{'}_{ij}B_{ij}}\Omega_{ij}&\leq \sqrt{\left(\bsum{j=1}{n}A^{'}_{ij}\right)\left(\bsum{j=1}{n}B_{ij}\Omega_{ij}^2\right) }  = \sqrt{a_i\bsum{j=1}{n}B_{ij}\Omega_{ij}^2 } \\
    &= \frac{\sqrt{a_i}\bsum{j=1}{n}B_{ij}\Omega_{ij}^2}{\sqrt{\left(\bsum{j=1}{n}B_{ij}\Omega_{ij}^2\right) }}  =\bsum{j=1}{n}\sqrt{\frac{a_iB_{ij}\Omega_{ij}^2}{\bsum{k=1}{n} B_{ik}\Omega_{ik}^2}B_{ij}}\Omega_{ij}=\bsum{j=1}{n}\sqrt{E_{ij}B_{ij}}\Omega_{ij}.
\end{align}
Therefore, 
\begin{align}
F(A^{'},B)&=\bsum{i=1}{m}\bsum{j=1}{n}\sqrt{A^{'}_{ij}B_{ij}}\Omega_{ij}\leq\bsum{i=1}{m}\bsum{j=1}{n}\sqrt{E_{ij}B_{ij}}\Omega_{ij}=F\left(T_1(A,B)\right).
\end{align}
Note that this inequality is strict unless for each $1\leq i \leq m$ such that \begin{equation}\sum\limits_{k=1}^{n} B_{ik}\Omega_{ik}^2>0,\end{equation}
$\{A^{'}_{ij}\}$ is a scalar multiple of $\{B_{ij}\Omega_{ij}^2\}$. Since we know $\bsum{j=1}{n}A^{'}_{ij}=a_i$, equality holds if and only if 
  \begin{equation}
      \{A^{'}_{ij}\}=\left\{\frac{a_i B_{ij}\Omega_{ij}^2}{\sum_{k=1}^{n}B_{ik}\Omega_{ik}^2}\right\}=\{E_{ij}\}.
  \end{equation}
The proof of (2) follows by a symmetric argument. 
\end{proof}

Finally, we need to observe the effect of $T_1$ and $T_2$ on the value of $F$ at the limit points of our sequence. This is not immediate because our previous results require that our semi-coupling is in the interior of $\tilde{\mathcal{A}}$, but the limit point of our sequence could lie on the boundary of $\tilde{\mathcal{A}}$.

\begin{lemma}\label{fixedpoint}
If  $(\overline{A},\overline{B})$ is a limit point of $(A,B)_{k}$, then \begin{equation}F(T_1(\overline{A},\overline{B}))=F(T_2(T_1(\overline{A},\overline{B})))=F(\overline{A},\overline{B}).\end{equation}
\end{lemma}

\begin{proof}
Let  $(\overline{A},\overline{B})$ be a limit point of $(A,B)_{k}$ and $(A,B)_{k_j}$ be a subsequence that converges to $(\overline{A},\overline{B})$.  Let $\epsilon > 0$.  Since $F\circ T_1$ is continuous, there exists a $\delta>0$ such that if $\|(A^{'},B^{'})-(\overline{A},\overline{B})\|<\delta$ then $|F(T_1(A^{'},B^{'}))-F(T_1(\overline{A},\overline{B}))|<\epsilon/3$.

Since $(A,B)_{k_j}$ converges to $(\overline{A},\overline{B})$, it follows from the continuity of $F$ that $F((A,B)_{k_j})$ converges to $F(\overline{A},\overline{B})$. Thus, there exists $k_j$ such that 
\begin{equation}
    \|(A,B)_{k_j}-(\overline{A},\overline{B})\|<\delta\text{ and }|F((A,B)_{k_j})-F(\overline{A},\overline{B})|<\epsilon/3.
\end{equation}
Recall, $F((A,B)_{k_j})\leq F(T_1((A,B)_{k_j})) \leq F(\overline{A},\overline{B})$. Thus, \begin{equation}
    |F(T_1((A,B)_{k_j}\}))-F((A,B)_{k_j})|\leq|F((A,B)_{k_j})-F(\overline{A},\overline{B})|<\epsilon/3.
\end{equation}
Therefore,

\begin{align}
  &  |F(T_1(\overline{A},\overline{B}))-F(\overline{A},\overline{B})|\leq|F(T_1(\overline{A},\overline{B}))-F(T_1((A,B)_{k_j}))|\\
   & \qquad\qquad +|F(T_1((A,B)_{k_j}))-F((A,B)_{k_j})|+|F((A,B)_{k_j})-F(\overline{A},\overline{B})|\\
    &\qquad<\epsilon/3+\epsilon/3+\epsilon/3=\epsilon
\end{align}
The proof that $F(T_2(T_1(\overline{A},\overline{B})))=F(\overline{A},\overline{B})$ follows by a similar argument.
\end{proof}

We can now proceed with the proof of Theorem ~\ref{convergence}.
\begin{proof}[Proof of Theorem  ~\ref{convergence}]
Note that for every $k\in \N$, $F((A,B)_{k+1})\geq F((A,B)_{k})$. Since $F$ is bounded above. The sequence defined by  $F((A,B)_{k})$ converges.

Let $(\overline{A},\overline{B})$ be a limit point of $(A,B)_{k}$. From  Lemma~\ref{decreasing}, we have $F(T_1((A,B)_k))\geq F(A^{'},B_k),$ for each  $(A^{'},B_k) \in \tilde{\mathcal A}$  where $B_k$ is fixed. Taking the limit of this inequality as $k\to \infty$, we obtain $F(T_1(\overline{A},\overline{B}))\geq F(A^{'},\overline{B}),$  for each $(A^{'},\overline{B}) \in \tilde{\mathcal A}$ where $\overline{B}$ is fixed. By Lemma~\ref{fixedpoint}, $F(T_1(\overline{A},\overline{B}))=F(\overline{A},\overline{B})$.

By the optimality condition (Prop. 2.1.2 in Section 2.1 of ~\cite{bertsekas1997nonlinear}) on convex sets, $\nabla_1 F(\overline{A},\overline{B}) (\overline{A}-A^{'})\leq 0$ for all possible $A^{'}$. Here $\nabla_1 F$ is the gradient of $F$ with respect to the first block. A similar argument shows that, $\nabla_2 F(\overline{A},\overline{B}) (\overline{B}-B^{'})\leq 0$ for all possible $B^{'}$.  Here $\nabla_2 F$ is the gradient of $F$ with respect to the second block.

Combining the inequalities and using the product structure of $\tilde{\mathcal A}$, we obtain the statement $\nabla F(\overline{A},\overline{B}) ((\overline{A},\overline{B})-(A^{'},B^{'}))\leq 0$ for all $(A^{'},B^{'})$ in $\tilde{\mathcal A}$. Thus by the optimality condition of concave functions on convex sets, $(\overline{A},\overline{B}) =\argmax\limits_{(A,B)\in\tilde{\mathcal A}}F(A,B).$
\end{proof}

\subsection{Implementation and Experiments}\label{implementation}
The theory developed in the previous section directly gives rise to an algorithm for computing the WFR distance. We outline this procedure in Algorithm~\ref{alg:WFR_iter}.
A PyTorch implementation of our algorithm is open source available at
\begin{center}
\url{https://github.com/emmanuel-hartman/WassersteinFisherRaoDistance}  
\end{center}
Unlike the Sinkhorn-type methods proposed in ~\cite{chizat2018scaling} and implemented in~\cite{flamary2021pot}, our method computes the Wasserstein-Fisher-Rao distance without any entropic regularization.In this implementation, we assume that for all $j$ there exists $i$ such that $\Omega_{ij}>0$ and  for all $i$ there exists $j$ such that $\Omega_{ij}>0$. As such, the zeroth rows and columns of the optimal discrete semi-coupling will have all zero entries, so we omit them from our implementation. Our reason for making this assumption on $\Omega_{ij}$ is that our main application of this algorithm is to compute SRNF distances between closed surfaces. For closed surfaces, this assumption will always be valid. (See Section 3 of this paper for more details.) 

We implemented Alg.~\ref{alg:WFR_iter} using PyTorch and perform the computations on the GPU. The main operations in our algorithm consist of element-wise matrix multiplication, which leads to a quadratic complexity. This can be also observed in the computation times in Table~\ref{fig:ex_meas}. As the obtained semi-coupling matrices are usually sparsely populated, we expect that an implementation utilizing sparse matrix data types could significantly improve the performance of the implementation.

\begin{algorithm}
\caption{Calculate $\operatorname{WFR}$}\label{alg:WFR_iter}
\begin{algorithmic} 
\Procedure{WFR\_distance}{$a,b,u,v,\epsilon$}

\State $\Omega_{ij}\gets(u_i\cdot v_j)$ if $i,j\geq 1$ and 1 otherwise
\State $A,B\gets \text{some intitialization in }\tilde{A}^{\mathrm{o}}$, $err\gets\infty$
\State $c\gets \sum_{i,j}\sqrt{A_{ij}\cdot B_{ij}}\cdot \Omega_{ij}$
\While $err\geq \epsilon$
\State $(A,B)\gets\Call{WFR\_iteration}{a,b,\Omega,A,B}$
\State $c_{\text{prev}}\gets c$
\State $c\gets \sum_{i,j}\sqrt{A_{ij}\cdot B_{ij}}\cdot \Omega_{ij}$
\State $err\gets (c-c_{\text{prev}})/c$
\EndWhile
\State \Return $\left(\sum_i a_i + \sum_j  b_j - 2\sum_{i,j}\sqrt{A_{ij}\cdot B_{ij}}\cdot \Omega_{ij} \right)^{1/2}$
\EndProcedure

\Procedure{WFR\_iteration}{$a,b,\Omega,A,B$}
 \State $B'_{ij}\gets B_{ij}\cdot\Omega_{ij}\cdot\Omega_{ij}$
\State $R_i\gets a_i/\sum_jB'_{ij} $ 
\State $A_{ij}\gets B'_{ij}\cdot R_i $
\State $A'_{ij}\gets A_{ij}\cdot\Omega_{ij}\cdot\Omega_{ij}$
\State $C_j\gets b_j/\sum_iA'_{ij}$
\State $B_{ij}\gets A'_{ij}\cdot C_j$
\State \Return(A,B)
\EndProcedure
\end{algorithmic}
\end{algorithm}

To quantify the performance of our algorithm we perform experiments comparing our method with entropy regularized methods using a small regularization parameter. 
Therefore we construct random pairs of finitely supported measures with a fixed number of support points. We then calculate the distance using both our algorithm and the unbalanced Sinkhorn algorithm~\cite{chizat2018scaling}, where we choose the regularization parameter to be $1\times10^{-3}$ (this was the smallest value that led to a stable performance across all experiments).  Therefore we will first discuss the relation between the variables of these two algorithms. Due to \cite{chizat2018unbalanced}, we have 
\begin{equation}
    \operatorname{WFR}^2(\mu,\nu)=\inf_{\gamma\in\mathcal{M}(M^2)} G(\gamma)
\end{equation}
where \begin{equation}G(\gamma) = \operatorname{KL}((\operatorname{Proj}_0)_\#\gamma,\mu)+\operatorname{KL}((\operatorname{Proj}_1)_\#\gamma,\nu)-\int_{M^2} \log(\cos^2(d(x,y)\wedge(\pi/2))) d\gamma(x,y).
\end{equation}
Forthcoming results of by Gallouët, Ghezzi and Vialard~\cite{gallouet2021regularity}, show that given a semi-coupling $(\gamma_1,\gamma_2)$ one can produce an optimal transport plan, 
\begin{equation}
    (\gamma_1,\gamma_2)\mapsto \gamma=\sqrt{\gamma_1\gamma_2}e^{-c/2} \text{ where }c=-\log(\cos^2(d(x,y)\wedge(\pi/2)))
\end{equation} such that 
\begin{equation}
    G(\gamma)=\int_{M^2}\left|\sqrt{\frac{\gamma_1}{\gamma}(u,v)}u-\sqrt{\frac{\gamma_2}{\gamma}(u,v)}v\right|^2 d\gamma(u,v)
\end{equation} for $\gamma$ such that $\gamma_1,\gamma_2\ll\gamma$. Therefore, we can compare the transport plans produced by both methods. The Sinkhorn type algorithm proposed in \cite{chizat2018scaling} then solves a regularized optimization problem given by \begin{equation}
    \inf\limits_{\gamma\in\mathcal{M}(M^2)}G(\gamma)+ \lambda\operatorname{Ent}(\gamma)
\end{equation} where $\lambda$ is the regularization parameter and $\operatorname{Ent}(\cdot)$ is an entropic regularization term. To obtain a fair comparison of the solutions obtained with our algorithm and the solutions obtained with the Sinkhorn algorithm, we disregarded the final entropy of the Sinkhorn solution and only compared the corresponding  transport costs. The distances that resulted from our algorithm were consistently smaller (and consequently more precise) across all experiments, which can be seen in Table~\ref{fig:ex_meas}, where we report the mean errors and variances of the Sinkhorn algorithm as compared to the obtained distance using our algorithm. For each number of support points, the relative errors were calculated by repeating the experiments 100 times. As one can see in this table our method produces significantly more accurate distances and without having to choose an entropic regularization parameter.

We also report the corresponding mean computation times for these experiments using the Sinkhorn algorithm. As the implementation of~\cite{flamary2021pot} does not utilize the GPU, we ported their implementation to PyTorch to be able to have a fair comparison of the two methods.  
Both algorithms are run on an Intel 3.2 GHz CPU with a Gigabyte GeForce GTX 2070 1620 MHz GPU. We used a maximum of 2000 iterations, but we usually observed a much faster convergence. While the GPU Sinkhorn implementation seems to scale better for larger numbers of support points, we believe that the above mentioned adaptions would lead to a similar complexity for Algorithm~\ref{alg:WFR_iter}.
In addition, as one can see in Figure \ref{fig:compare_sinkhorn}, the Sinkhorn algorithm has a significantly faster convergence and will thus lead to a faster computation time. We want to emphasize, however, that our algorithm solves the exact problem, while the Sinkhorn algorithm only tackles a regularized problem. This is also mirrored by the fact that in all our experiments the solutions obtained with our algorithm have a lower distance as compared to corresponding Sinkhorn solutions, cf.~Table~\ref{fig:ex_meas}.

\begin{table}[htbp]
\centering
\begin{tabular}{|c||c|c|c|c|c|c|c|}
\hline
Support&\multicolumn{2}{c|}{ WFR Distance}&\multicolumn{2}{c|}{Sinkhorn Error}&\multicolumn{3}{c|}{Timing in seconds}\\
\cline{2-8}
points &Alg.~\ref{alg:WFR_iter}&Sink.&Mean&Variance&Alg.~\ref{alg:WFR_iter}&Sink. (GPU)&Sink. \cite{flamary2021pot}\\\hline\hline
128&\textbf{24.045}&25.040&3.883\%&1.835\%&0.100&1.348&0.101\\\hline
256&\textbf{25.824}&26.951&4.445\%&2.909\%&0.120&1.413&0.135\\\hline
512&\textbf{27.526}&29.532&7.303\%&3.071\%&0.179&1.470&0.219\\\hline
1024&\textbf{29.724}&32.302&8.673\%&3.490\%&0.398&1.523&0.509\\\hline
2048&\textbf{31.728}&34.533&8.929\%&3.976\%&1.372&1.901&4.506\\\hline
4096&\textbf{35.061}&38.027&8.382\%&3.699\%&5.569&3.152&20.323\\\hline
8192&\textbf{38.385}&40.426&5.250\%&3.154\%&23.472&8.732&81.905\\\hline
\end{tabular}
\caption{Comparison of the Sinkhorn algorithm of~\cite{chizat2018scaling} and our Algorithm. For each number of support points theses results were obtained by calculating the WFR distance between 100 pairs of randomly chosen measures.}\label{fig:ex_meas}
\end{table}

\begin{figure}[htbp]
    \centering
    \includegraphics[height=0.29\textwidth]{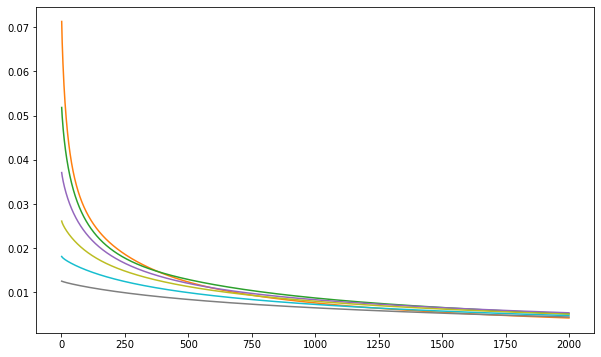}
    \includegraphics[height=0.29\textwidth]{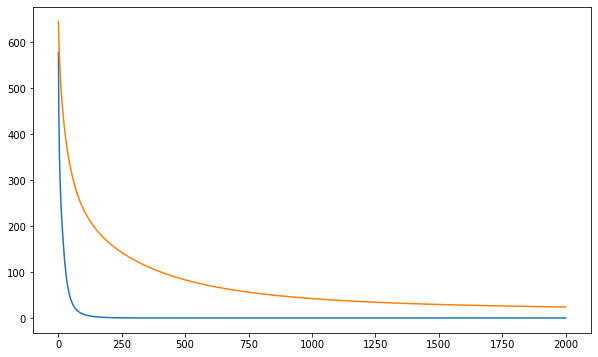}
    \caption{\textbf{Left:} Plot of $||(A,B)-(A^*,B^*)||_{L^2}$ for 2000 iterations on pairs of measures with $n$ supports:  256 (Orange), 512 (Green), 1024 (Purple), 2048 (Olive), 4096 (Light Blue), 8192 (Grey). \textbf{Right:} Plot of $||\gamma-\gamma^*||_2$ for 2000 iterations of our algorithm (Orange) and the Sinkhorn algorithm (Blue). The optimal semi-couplings, $(A^*,B^*)$, and optimal plans $\gamma^*$ are computed with 10,000 iterations of the respective algorithms.}
    \label{fig:compare_sinkhorn}
\end{figure}

\section{The SRNF shape metric as an unbalanced transport problem} 
In this section, we will present the main result of our article: the interpretation of the SRNF shape distance as an unbalanced OMT problem. Our result will then allow us to compute the SRNF distance using the algorithm introduced in Section~\ref{algorithm} and we will use this to present several numerical examples at the end of the section.

\subsection{Shape spaces of surfaces}
In all of this section let $M$ be a smooth, connected, compact, oriented Riemannian 2-dimensional manifold with or without boundary. In addition to the smooth structure on $M$, we will be also interested in a piecewise linear structure on it. By ~\cite{whitehead1940c1} any such $M$ indeed admits a Whitehead PL structure. That is, there exists a polyhedral surface $K$ in $\R^3$ and a homeomorphism called a triangulation $\sigma:K\to M$ such that $\sigma$ is differentiable with injective differential on each face of $K$. 

We denote the space of all Lipschitz immersions of $M$ into $\reals^3$ by $\Imm_{\operatorname{Lip}}(M,\reals^3)$, i.e.,
\begin{align}
\Imm_{\operatorname{Lip}}(M,\reals^3)=\{f\in W^{1,\infty}(M,\reals^3): Tf \text{ is inj. a.e.}\}\;.    
\end{align}
The reason for considering immersions of the Lipschitz class is that this space has two important subsets: the space of smooth immersions from $M$ to $\reals^3$ and the space of functions $f\in \Imm_{\operatorname{Lip}}(M,\reals^3)$ that are PL with respect to the given PL structure on $M$.

As we are interested in unparametrized surfaces, we have to factor out the action of the group of diffeomorphisms. In the context of Lipschitz immersions the natural group of reparametrizations for us to consider is the group of all 
orientation preserving, bi-Lipschitz diffeomorphisms:
 \begin{equation}
     \DiffLip(M)=\{ \gamma\in W^{1,\infty}(M,M):\; \gamma^{-1}\in W^{1,\infty}(M,M), |D\gamma|>0 \text{ a.e.}\},
     \end{equation} 
where  $|D\gamma|$ denotes the Jacobian determinant of $\gamma$, which is well-defined as $D\gamma\in L^{\infty}$. For reasons, that will become clear later in Section~\ref{subsec:SRNF}, we will also consider two subsets of $\DiffLip(M)$, namely the group of smooth, orientation preserving diffeomorphisms 
\begin{align}
  \DiffSmooth(M)=\{\gamma\in C^{\infty}(M,M):\; \gamma^{-1}\in C^{\infty}(M,M), |D\gamma|>0\},
\end{align}
and the set of $PL$-homeomorphisms. To define the latter, we recall that a PL-homeomorphism on $M$ is a homeomorphism $\gamma: M \to M$, such that 
there exists some subdivision of $K$ such that $\sigma^{-1} \circ \gamma \circ \sigma$ is linear on each face. We denote the corresponding space of all orientation preserving homeomorphisms by
 \begin{equation}
     \DiffPL(M)=\{\gamma\in \operatorname{Hom}_{PL}(M,M):\, |D\gamma|>0 \text{ a.e.} \}.
 \end{equation}
Note that any of these reparametrization groups act by composition from the right on $\ImmLip(M,\reals^3)$. In addition to the action of these reparametrization groups, we also want to identify surfaces that only differ by a translation. 
This leads us to consider the following three quotient spaces:
\begin{align}
&\mathcal S_{\operatorname{\infty}}:=\ImmLip(M,\mathbb R^3)/\DiffSmooth(M)/\operatorname{trans}  \\ 
&\mathcal S_{\operatorname{Lip}}:=\ImmLip(M,\mathbb R^3)/\DiffLip(M)/\operatorname{trans} \\
&\mathcal S_{\operatorname{PL}}:=\ImmLip(M,\mathbb R^3)/\DiffPL(M)/\operatorname{trans} ,
\end{align}
which will play a central role in the remainder of the article. Note that we always consider immersions of Lipschitz class and only vary the regularity of the group acting on this space.

\subsection{The SRNF framework}\label{subsec:SRNF}  
The square root normal field (SRNF) map was first introduced by Jermyn et al. 
in ~\cite{jermyn2012elastic} for the space of smooth immersions. As we will see in the following this mapping directly extends to all immersions of the Lipschitz class. 

For any given $f\in\ImmLip(M,\reals^3)$, the orientation on $M$ allows us to consider the unit normal vector field $n_f:M\to\reals^3$, which is well-defined as a function in $L^{\infty}(M,\reals^3)$. Furthermore, let $\{v,w\}$ be an orthonormal basis of $T_xM$. Then for any $f\in\ImmLip(M,\reals^3)$ we can also define the area multiplication factor at $x\in M$ via $a_f(x)=|df_x(v)\times df_x(w)|$. 

The SRNF map is then given by  
\begin{align}
\Phi:\ImmLip(M,\reals^3)/\operatorname{translations}&\to L^2(M,\reals^3)\\
f&\mapsto q_f,
\end{align}
where
\begin{equation}q_f(x)=\sqrt{a_f(x)}\;n_f(x).\end{equation}
We can now use the SRNF to define a pseudometric on $\ImmLip(M,\reals^3)/\operatorname{translations}$ by \begin{equation}\label{distfuncdef}
d_{\Imm}(f_1,f_2)=\|\Phi(f_1)-\Phi(f_2)\|_{L^2}.
\end{equation} The function $d_{\Imm}$ is only a pseudo-metric due to the non-injectivity of $\Phi$. Examples of this degeneracy have been discussed extensively in the recent article ~\cite{klassen2020closed}.

Next we consider a right-action of $\DiffLip(M)$ on $L^2(M,\reals^3)$ 
that is compatible with the mapping $\Phi$.
Therefore we let
\begin{align}
 (g*\gamma)(x)=\sqrt{a_\gamma(x)}g(\gamma(x)),
\end{align} where $a_\gamma(x)$, the area multiplication factor of $\gamma$ at $x$, is defined by $a_\gamma(x)=\sqrt{|D_\gamma(x)|}$. It is easy to check that \begin{align}
   \Phi(f)*\gamma=\Phi(f\circ\gamma) 
\end{align} 
and that the action of $\DiffLip(M)$ on $L^2(M,\reals^3)$ is by linear isometries, if we put the usual $L^2$ inner product on $L^2(M,\reals^3)$. Thus it follows that the SRNF pseudometric on $\ImmLip(M,\reals^3)$ is invariant with respect to this action and thus it descends to a pseudometric on the quotient space $\mathcal S_{\operatorname{Lip}}$, which is given by
\begin{equation}
   d_{\mathcal S_{\operatorname{Lip}}}([f_1],[f_2])=\inf\limits_{\gamma\in\DiffLip(M)}d(f_1,f_2\circ\gamma),\qquad [f_1],[f_2]\in \mathcal{ S}_{\operatorname{Lip}}(M)
\end{equation}
Since $\DiffSmooth(M)$ and $\DiffPL(M)$ are both subsets of $\DiffLip$ the invariance properties continue to hold for the actions of the smaller groups and we can also consider the corresponding distance functions
\begin{align}
d_{\mathcal S_{\infty}}([f_1],[f_2])=\inf\limits_{\gamma\in\DiffSmooth(M)}d(f_1,f_2\circ\gamma),\qquad [f_1],[f_2]\in \mathcal{S}_{\infty}(M)\\
   d_{\mathcal S_{\operatorname{PL}}}([f_1],[f_2])=\inf\limits_{\gamma\in\DiffPL(M)}d(f_1,f_2\circ\gamma),\qquad [f_1],[f_2]\in \mathcal{S}_{\operatorname{PL}}(M)
\end{align}
 The remainder of this section will be devoted to showing that the SRNF distance for each of these three group actions is equivelent. More precisely, we aim to prove the following theorem:
 \begin{theorem}\label{thm:PLequiv}
 Let $f_1,f_2\in \ImmLip(M,\reals^3)$. Then 
\begin{equation}
 d_{\mathcal S_{\operatorname{Lip}}}([f_1],[f_2])= 
 d_{\mathcal S_{\infty}}([f_1],[f_2])=
 d_{\mathcal S_{\operatorname{PL}}}([f_1],[f_2]).
\end{equation}
\end{theorem}
 \begin{remark}
 Note that this result implies in particular that for smooth immersions $f_1,f_2$ the SRNF metric as defined in this section is equal to the SRNF metric considered in~\cite{jermyn2012elastic}.
 \end{remark}
 The main ingredient for proving this theorem will be the following lemma  concerning the continuity of the action of $\DiffLip(M)$ on $L^2(M,\mathbb{R}^3)$. Therefore we first  note that $\DiffLip(M)\subseteq W^{1,\infty}(M,M)\subseteq W^{1,2}(M,M)$  and thus we can equip $\DiffLip(M)$ with the $\|\cdot\|_{W^{1,2}}$ norm. 
 \begin{lemma}\label{lem:cont_act}
 The map 
 \begin{align}
  \left(L^2(M,\mathbb{R}^3),\|\cdot\|_{L^2}\right)\times\left(\DiffLip(M),\|\cdot\|_{W^{1,2}}\right)&\to  \left( L^2(M,\mathbb{R}^3),\|\cdot\|_{L^2}\right)  \\
  (q,\gamma) &\mapsto q*\gamma
 \end{align} 
 is jointly continuous.
 \end{lemma}
 Since the $W^{1,\infty}$ topology dominates the $W^{1,2}$-topology, the lemma would also hold when equipping $\DiffLip(M)$ with the Lipschitz topology. 
 The reason for equipping it instead with the $W^{1,2}$-topology is that the subgroups of smooth and PL diffeomorphisms are dense with respect to this topology, but not with respect to the $W^{1,\infty}$-topology. The density of these groups will be a crucial ingredient for our proof of Theorem~\ref{thm:PLequiv}.

 \begin{proof} 
 The proof of this result is inspired by~\cite[Proposition 7]{bruveris2016optimal}, where a related result for one-dimensional domain space $M$ is shown.  \newline
 \textbf{Step 1 (Piecewise constant $q$)}
 Let $q\in L^2(M,\mathbb{R}^3)$ be piecewise constant, i.e., there exists a disjoint family of sets $M_j$ such that $q=\sum_{j=1}^{N}u_j\chi(M_j)$. Given a sequence $\gamma_n\to \gamma$ in $\left(\DiffLip(M),\|\cdot\|_{W^{1,2}}\right)$, we need to show $q*\gamma_n\to q*\gamma$ in $L^2(M,\mathbb{R}^3)$.  Since $M$ is a 2-manifold, the map
\begin{align}
    W^{1,2}(M,M)\to L^1(M,M)\text{ defined via }
    \gamma\mapsto a_{\gamma}
\end{align}
is continuous by the Sobolev multiplication theorem. Thus, 
\begin{equation}
    \int_M|a_{\gamma}-a_{\gamma_n}|dx\to 0. 
\end{equation}
Further this implies
\begin{equation}
    \int_M|a_{\gamma}-a_{\gamma_n}|dx= \int_M\left(\sqrt{|a_{\gamma}-a_{\gamma_n}|}\right)^2dx\geq\int_M|\sqrt{a_{\gamma}}-\sqrt{a_{\gamma_n}}|^2dx
\end{equation}
and therefore we also have 
\begin{equation}
    \int\limits_{M}\left|\sqrt{a_{\gamma}}-\sqrt{a_{\gamma_n}}\right|^2dx\to 0.
\end{equation}
 Next we define the sets $M_{i,n}^j=\gamma_n(M_j)\cap \gamma(M_i)$. Using this we can write the integral via a double sum as:
 \begin{align}
    \int_M|(q\circ\gamma)\sqrt{a_{\gamma}}-(q\circ\gamma_n)\sqrt{a_{\gamma_n}}|^2 dx = \sum_{i=1}^{N}\sum_{j=1}^{N}\int_{M_{i,n}^j}|u_j\sqrt{a_{\gamma}}-u_i\sqrt{a_{\gamma_n}}|^2dx
 \end{align}
 Note that $\gamma_n \to \gamma$ in $W^{1,\infty}$ and thus for $i\neq j$, $\mu(M_{i,n}^j)\to 0$. Meanwhile 
 \begin{align}
    \int_{M_{i,n}^j}|u_j\sqrt{a_{\gamma}}-u_i\sqrt{a_{\gamma_n}}|^2dx&\leq  \int_{M_{i,n}^j}|u_j\sqrt{a_{\gamma}}-u_i\sqrt{a_{\gamma}}|^2dx+\|q\|^2_\infty\int_{M_{i,n}^j}|\sqrt{a_{\gamma}}-\sqrt{a_{\gamma_n}}|^2dx\\
    &=|u_j-u_i|^2\int_{M_{i,n}^j}|a_{\gamma}|dx+\|q\|^2_\infty\int_{M_{i,n}^j}|\sqrt{a_{\gamma}}-\sqrt{a_{\gamma_n}}|^2dx.
 \end{align}
 Recall, $\gamma\in W^{1,\infty}(M,M)$ thus $|a_{\gamma}|$ is bounded. So \begin{equation}
     |u_j-u_i|^2\int\limits_{M_{i,n}^j}|a_{\gamma}|dx\to 0
 \end{equation}
 as $\mu(M_{i,n}^j) \to 0$.
 Additionally,
 \begin{equation}
     \int\limits_{M}\left|\sqrt{a_{\gamma}}-\sqrt{a_{\gamma_n}}\right|^2dx\to 0\text{ therefore when }i\neq j, \,\int_{M_{i,n}^j}|u_j\sqrt{a_{\gamma}}-u_i\sqrt{a_{\gamma_n}}|^2dx\to 0.
 \end{equation}
In the case where $i=j$ we have 
 \begin{equation}
     \int_{M_{i,n}^i}|u_i\sqrt{a_{\gamma}}-u_i\sqrt{a_{\gamma_n}}|^2=\int_{M_{i,n}^i}|u_i|^2\cdot|\sqrt{a_{\gamma}}-\sqrt{a_{\gamma_n}}|^2dx\leq \|q\|^2_\infty \int_{M_{i,n}^i}|\sqrt{a_{\gamma}}-\sqrt{a_{\gamma_n}}|^2dx.
 \end{equation}
 Thus, 
 \begin{equation}
     \int_M|(q\circ\gamma)\sqrt{a_{\gamma}}-(q\circ\gamma_n)\sqrt{a_{\gamma_n}}|^2 dx \to 0.
 \end{equation} 
\textbf{Step 2. (general $q$)} 
Let now $q\in L^2(M,\mathbb{R}^3)$, $\gamma_n\to \gamma$ in $\DiffLip(M)$, and $\epsilon>0$. Pick $q'$ to be piecewise constant such that $\|q-q'\|_{L^2}<\epsilon/3$. Then using the fact that $\DiffLip(M)$ acts by isometries and the triangle inequality we have
\begin{align}
    \|q*\gamma-q*\gamma_n\|_{L^2}&\leq \|q*\gamma-q'*\gamma\|_{L^2}+ \|q'*\gamma-q'*\gamma_n\|_{L^2}+ \|q'*\gamma_n-q*\gamma_n\|_{L^2}\\&\leq
    \|q-q'\|_{L^2}+ \|q'*\gamma-q'*\gamma_n\|_{L^2}+ \|q'-q\|_{L^2}\\&
    \leq \epsilon/3+\|q'*\gamma-q'*\gamma_n\|_{L^2}+\epsilon/3.
\end{align}
 From here we can conclude the result using Step 1.
 
\noindent\textbf{Step 3. (joint continuity)} 
Let now $q_n\to q$ in $L^2(M,\mathbb{R}^3)$, $\gamma_n\to \gamma$ in $\DiffLip(M)$, and $\epsilon>0$. Using the fact that $\DiffLip(M)$ acts by isometries and the triangle inequality we have
\begin{align}
    \|q*\gamma-q_n*\gamma_n\|_{L^2}&\leq \|q*\gamma-q*\gamma_n\|_{L^2}+ \|q*\gamma_n-q_n*\gamma_n\|_{L^2}\\&=
    \|q*\gamma-q*\gamma_n\|_{L^2}+ \|q-q_n\|_{L^2}\\
\end{align}
 From here we can conclude the result using Step 2.
\end{proof}
We are now able to prove Theorem~\ref{thm:PLequiv}:
\begin{proof}[Proof of Theorem~\ref{thm:PLequiv}]
Given $[f_1],[f_2]\in \mathcal S_{\operatorname{Lip}}$, choose parametrized representations  $f_1,f_2\in\ImmLip(M,\reals^3)$.

\noindent\textbf{Part 1.} $(\boldsymbol{d_{\mathcal S_{\operatorname{Lip}}}([f_1],[f_2])= 
 d_{\mathcal S_{\infty}}([f_1],[f_2]))}$ Note that $\DiffSmooth(M)$ is a subset of $\DiffLip(M)$. Thus,  \begin{equation}
    \inf\limits_{\gamma\in\DiffLip(M)}\|\Phi(f_1)-\Phi(f_2\circ\gamma)\|_{L^2}\leq\inf\limits_{\gamma\in\DiffSmooth(M)}\|\Phi(f_1)-\Phi(f_2\circ\gamma)\|_{L^2}.
\end{equation}
To show the opposite inequality, take an arbitrary $g\in\DiffLip(M)$ and let $\epsilon>0$. By Lemma~\ref{lem:cont_act}, there exists some $\delta>0$ such that $\|g-\gamma\|_{W^{1,2}}<\delta$ implies $\|\Phi(f_2)*g-\Phi(f_2)*\gamma\|<\epsilon$. Since $M$ is a 2-manifold, $C^\infty(M,M)$ maps are dense in the Sobolev functions $W^{1,2}(M,M)$ ~\cite{bethuel1988density}. Since $\DiffLip(M)\subseteq W^{1,2}(M,M)$, there exists $\gamma\in \DiffSmooth(M)$ such that $\|g-\gamma\|_{W^{1,2}}<\delta$. Thus, 
\begin{align}
    \|\Phi(f_1)-\Phi(f_2)*\gamma\|_{W^{1,2}}&\leq \|\Phi(f_1)-\Phi(f_2)*g\|_{W^{1,2}}+\|\Phi(f_2)*g-\Phi(f_2)*\gamma\|_{W^{1,2}}\\&<\|\Phi(f_1)-\Phi(f_2)*g\|_{W^{1,2}}+\epsilon
\end{align}
Therefore, \begin{equation}
    \inf\limits_{\gamma\in\DiffLip(M)}\|\Phi(f_1)-\Phi(f_2\circ\gamma)\|_{L^2}=\inf\limits_{\gamma\in\DiffSmooth(M)}\|\Phi(f_1)-\Phi(f_2\circ\gamma)\|_{L^2}.
\end{equation}
\textbf{Part 2.} $(\boldsymbol{d_{\mathcal S_{\operatorname{Lip}}}([f_1],[f_2])= 
 d_{\mathcal S_{PL}}([f_1],[f_2]))}$ 
 This follows exactly as in Part 1, using that $\DiffPL(M)$ is a dense subset of $\DiffLip(M)$. To prove the density we use the fact that $\operatorname{Hom}_{PL}(M,M)$ maps are dense in the Sobolev functions $W^{1,2}(M,M)$ as $\operatorname{dim}(M)=2$, cf. ~\cite{van2014approximation}. 
\end{proof}

\subsection{SRNF shape metric as an unbalanced optimal transport problem.}
First we define a map from  $L^2(M,\mathbb{R}^3)$ to the the space of positive finite Borel measures on $S^2$, and then show that computing the shape distance between two surfaces is equivalent to computing the Wasserstein-Fisher-Rao distance between the corresponding measures. For $q\in L^2(M,\mathbb{R}^3)$ and $U\subseteq S^2$ open, define 
\begin{equation}
    q^*U=\{x\in M|q(x)\neq 0 \text{ and } q(x)/|q(x)| \in U \}.
\end{equation}
Then we can define
\begin{align}
     &\psi:L^2(M,\mathbb{R}^3)\to\mathcal{M}(S^2)\text{ via }\psi(q)=\mu_q\\ 
     &\text{where for } U\subseteq S^2, \mu_q(U)=\int_{q^*U}q(x)\cdot q(x)dx.
\end{align}
The main goal of this section is to show that the shape pseudo-distance in the SRNF framework can be written as a pullback of the Wasserstein-Fisher-Rao distance via the map $\psi$. We will show this result only for the dense subset of PL surfaces. We expect that a careful analysis of the continuity of the SRNF map would allow one to obtain this result also for general Lipschitz immersions, which we plan to study in future work. 

\begin{theorem}\label{thm:maintheorem}
Given two PL surfaces $S_1$ and $S_2$ parameterized by $f_1$ and $f_2$ the SRNF shape distance can be computed as an unbalanced transport problem. More precisely, we have  
\begin{equation}
    d_{\mathcal S_{\operatorname{Lip}}}([f_1],[f_2])=\operatorname{WFR}(\psi\circ\phi(f_1),\psi\circ\phi(f_2)).
    \end{equation}
\end{theorem} 
\begin{proof}
Assume that $S_1$ and $S_2$ are triangulated compact oriented PL surfaces in $\reals^3$, and let $f:M\to S_1$ be a PL parametrization of $S_1$. Let $\{\sigma_i\}_{i=1}^m$ denote the faces of $S_1$, and let $\{\tau_j\}_{j=1}^n$ denote the faces of $S_2$. For each $i$, assume that $\sigma_i$ has area $a_i$ and oriented unit normal vector $u_i$. For each $j$, assume that $\tau_j$ has area $b_j$ and  oriented unit normal vector $v_j$. Let $\mathcal{A}$ be the set of all discrete semi-couplings from $\psi(S_1)$ to $\psi(S_2)$. 

\begin{claim}\label{maintheoremclaim1} We may equivalently express the shape distance by
\begin{equation} d_{\mathcal S_{\operatorname{Lip}}}([f_1],[f_2])^2=\sum_{i=1}^ma_i+\sum_{j=1}^nb_j-2\sup_{g\in\operatorname{Hom}_{PL}^+(S_1,S_2)}\int_M \Phi(f_1)\cdot \Phi(g\circ f_1).
\end{equation}
\end{claim}

\noindent\textit{Proof of Claim \ref{maintheoremclaim1}: }Recall by Theorem~\ref{thm:PLequiv},
\begin{align}
    d_{\mathcal S_{\operatorname{Lip}}}([f_1],[f_2])^2&=\inf\limits_{\gamma\in\DiffPL(M)}\|\Phi(f_1)-\Phi(f_2\circ\gamma)\|^2_{L^2}\\
    &=\inf\limits_{\gamma\in\DiffPL(M)}\int\limits_{M} |\Phi(f_1)-\Phi(f_2\circ\gamma)|^2 dx\\
    &=\sum_{i=1}^ma_i+\sum_{j=1}^nb_j-2\sup\limits_{\gamma\in\DiffPL(M)}\int\limits_{M} |\Phi(f_1)\cdot \Phi(f_2\circ\gamma)|dx
\end{align}

Every $\gamma\in\DiffPL(M)$, corresponds to $g=f_2\circ \gamma \circ f_1^{-1}\in \operatorname{Hom}_{PL}^+(S_1,S_2)$. Thus,
\begin{equation}
d_{\mathcal S_{\operatorname{Lip}}}([f_1],[f_2])^2=\sum_{i=1}^ma_i+\sum_{j=1}^nb_j-2\sup_{g\in\operatorname{Hom}_{PL}^+(S_1,S_2)}\int_M \Phi(f_1)\cdot \Phi(g\circ f_1).
\end{equation}
This completes the proof of Claim \ref{maintheoremclaim1}. Now recall from Remark~\ref{reformulateSRMM}
\begin{equation}
    \operatorname{WFR}(\psi\circ\phi(f_1),\psi\circ\phi(f_2))^2=\sum_{i=1}^ma_i+\sum_{j=1}^nb_j-2\sup_{(A,B)\in\mathcal{A}}\sum_{i=1}^{m}\sum_{j=1}^{n}\sqrt{A_{ij}B_{ij}}(u_i\cdot v_j)  
\end{equation}
Therefore, showing $d_{\mathcal S_{\operatorname{Lip}}}([f_1],[f_2])=\operatorname{WFR}(\psi\circ\phi(f_1),\psi\circ\phi(f_2))$ is equivalent to showing \begin{equation}
    \sup_{g\in\operatorname{Hom}_{PL}^+(S_1,S_2)}\int_M \Phi(f_1)\cdot \Phi(g\circ f_1)dx=\sup_{(A,B)\in\mathcal{A}}\sum_{i=1}^{m}\sum_{j=1}^{n}\sqrt{A_{ij}B_{ij}}(u_i\cdot v_j).
\end{equation} 
\begin{claim}\label{ApproxHom}
 Assume that $(A,B)$ is a discrete semi-coupling from $\psi(S_1)$ to $\psi(S_2)$. Then for all $\epsilon>0$ there is a PL homeomorphism $g:S_1\to S_2$ such that 
\end{claim}
\begin{equation}\left|\int_M q_{f}\cdot q_{g\circ f}-\sum_{i,j}\sqrt{A_{ij}B_{ij}}(u_i\cdot v_j)\right|<\epsilon.\end{equation}

\noindent\textit{Proof of Claim \ref{ApproxHom}: }Let $(A,B)$ be a discrete semi-coupling from $\psi(S_1)$ to $\psi(S_2)$ such that for each $1\leq i\leq m$ and $1\leq j\leq n$, $A_{ij},B_{ij}>0$. We will first prove the claim for this restricted case and extend it to all semi-couplings by continuity.

Let $a:M\to \mathbb{R}$ denote the area multiplication factor of $f$, and let $u:M\to S^2$ denote the unit normal vector function corresponding to $f$. First we choose a real number $r\in(0,1)$. For each $1\leq i\leq m$, subdivide $\sigma_i$ into $n$ smaller 2-simplexes $\sigma_{ij}$ such that each  $\sigma_{ij}$ has area $A_{ij}$. Similarly, for each $1\leq j\leq n$, subdivide $\tau_j$ into $m$ smaller 2-simplexes $\tau_{ij}$ such that each $\tau_{ij}$ has area $B_{ij}$. For each $1\leq i\leq m$ and $1\leq j\leq n$, choose a smaller 2-simplex $\tilde\sigma_{ij}$, whose closure is contained in the interior of $\sigma_{ij}$, such that $\tilde\sigma_{ij}$ has area equal to $rA_{ij}$. Similarly, for each $1\leq i\leq m$ and $1\leq j\leq n$, choose a smaller 2-simplex $\tilde\tau_{ij}$, whose closure is contained in the interior of $\tau_{ij}$, such that $\tilde\tau_{ij}$ has area equal to $rB_{ij}$.

We now construct an orientation preserving PL homeomorphism $g_r:S_1\to S_2$. First, for each $1\leq i\leq m$ and $1\leq j\leq n$, define $g_r:\tilde\sigma_{ij}\to\tilde\tau_{ij}$ to be an arbitrary PL orientation preserving homeomorphism with constant area multiplication factor. Note that 
$S_1-\left(\bigcup\limits_{i=1}^m\bigcup\limits_{j=1}^n\tilde\sigma_{ij}^{\mathrm{o}}\right)$ is homeomorphic to $S_2-\left(\bigcup\limits_{i=1}^m\bigcup\limits_{j=1}^n\tilde\tau_{ij}^{\mathrm{o}}\right)$. Hence, we can simply extend the homeomorphism $g_r$ which we already defined on the $\tilde\sigma_{ij}$'s to a homeomorphism $S_1\to S_2$ in an arbitrary manner.  Denote the unit normal function  coming from the parametrization $g_r\circ f$ of $S_2$ by $v_r:M\to S^2$. Denote the area multiplication factor of $g_r\circ f$ by $m_r:M\to\reals$.

Write $M=M_1\cup M_2$, where $M_1=f^{-1}\left(\bigcup\limits_{i=1}^m\bigcup\limits_{j=1}^n\tilde\sigma_{ij}\right)$ and $M_2=\overline{M-M_1}$. For each $1\leq i\leq m$ and $1\leq j\leq n$, let $M_{ij}=f^{-1}(\tilde\sigma_{ij})$. Note that on each $M_{ij}$, $m_r=\frac{aB_{ij}}{A_{ij}}$. Compute:
\begin{align}
\int_{M_1}q_f\cdot q_{g_r\circ f}& =\sum\limits_{i=1}^{m}\sum_{j=1}^{n}\int_{M_{ij}}\sqrt{a}\sqrt{m_r}(u_i\cdot v_j)
=\sum\limits_{i=1}^{m}\sum_{j=1}^{n}\int_{M_{ij}}\sqrt{a}\sqrt{\frac{aB_{ij}}{A_{ij}}}(u_i\cdot v_j)\\
&=\sum\limits_{i=1}^{m}\sum_{j=1}^{n}\int_{M_{ij}}a\sqrt{\frac{B_{ij}}{A_{ij}}}(u_i\cdot v_j)
=\sum\limits_{i=1}^{m}\sum_{j=1}^{n}\operatorname{area}(\tilde\sigma_{ij})\sqrt{\frac{B_{ij}}{A_{ij}}}(u_i\cdot v_j)\\
&=\sum\limits_{i=1}^{m}\sum_{j=1}^{n}\sqrt{\operatorname{area}(\tilde\sigma_{ij})}\sqrt{\frac{\operatorname{area}(\tilde\sigma_{ij})B_{ij}}{A_{ij}}}(u_i\cdot v_j)\\
&=\sum\limits_{i=1}^{m}\sum_{j=1}^{n}\sqrt{\operatorname{area}(\tilde\sigma_{ij})}\sqrt{\operatorname{area}(\tilde\tau_{ij})}(u_i\cdot v_j)\\
&=\sum\limits_{i=1}^{m}\sum_{j=1}^{n}\sqrt{rA_{ij}}\sqrt{rB_{ij}}(u_i\cdot v_j).
\end{align}
Meanwhile by the Schwarz inequality,
\begin{align}
    \left|\int_{M_2}q_f\cdot q_{g_r\circ f}\right|
    &=\left|\int_{M_2}\sqrt{a}\sqrt{m_r}(u\cdot v_r)\right|\leq\int_{M_2}\sqrt{a}\sqrt{m_r}\left| u\cdot v_r \right|\\
    &\leq\int_{M_2}\sqrt{a}\sqrt{m_r}\leq\sqrt{\int_{M_2}a}\sqrt{\int_{M_2}m_r}\\
    &=\sqrt{\mbox{area}\left(S_1-\left(\bigcup\limits_{i=1}^m\bigcup\limits_{j=1}^n\tilde\sigma_{ij}^{\mathrm{o}}\right)\right)}\sqrt{\mbox{area}\left(S_2-\left(\bigcup\limits_{i=1}^m\bigcup\limits_{j=1}^n\tilde\tau_{ij}^{\mathrm{o}}\right)\right)}\\
    &=\sqrt{(1-r)\operatorname{area}(S_1) }\sqrt{(1-r)\operatorname{area}(S_2) }.
\end{align}
So as we let $r\to 1$,
\begin{equation}
    \int_{M_1}q_f\cdot q_{g_r\circ f}\to\sum\limits_{i=1}^{m}\sum_{j=1}^{n}\sqrt{A_{ij}B_{ij}}(u_i\cdot v_j) \operatorname{ and }
\int_{M_2}q_f\cdot q_{g_r\circ f}\to 0.
\end{equation}
Hence, 
\begin{equation}
    \int_{M}q_f\cdot q_{g_r\circ f}\to\sum\limits_{i=1}^{m}\sum_{j=1}^{n}\sqrt{A_{ij}B_{ij}}(u_i\cdot v_j). 
\end{equation}Thus Claim \ref{ApproxHom} follows for the case in which for each $1\leq i\leq m$ and $1\leq j\leq n$, $A_{ij}>0$ and $B_{ij}>0$. The general case then follows immediately from the continuity of
\begin{equation}
    \sum\limits_{i=1}^{m}\sum_{j=1}^{n}\sqrt{A_{ij}B_{ij}}(u_i\cdot v_j)
\end{equation} as a function of $(A,B)$.
This completes the proof of Claim \ref{ApproxHom}. It follows that \begin{equation}
    \sup_{g\in\operatorname{Hom}_{PL}^+(S_1,S_2)}\int_M q_f\cdot q_{g\circ f}\geq\sup_{\{A_{ij}\},\{B_{ij}\}\in\mathcal{A}}\sum_{i=1}^{m}\sum_{j=1}^{n}\sqrt{A_{ij}B_{ij}}(u_i\cdot v_j).
\end{equation}
We are left to show the opposite inequality.
\begin{claim}\label{maintheoremclaim3}
 Assume $g$ is a PL-homeomorphism from $S_1$ to $S_2$, then there exists a discrete semi-coupling $(A,B)$ such that \begin{equation}\int_M q_f\cdot q_{g\circ f}\leq \sum_{i,j}\sqrt{A_{ij}}\sqrt{B_{ij}}(u_i\cdot v_j).\end{equation}
\end{claim}
\noindent\textit{Proof of Claim \ref{maintheoremclaim3}: }Let $g:S_1\to S_2$ be an orientation preserving PL homeomorphism. For $1\leq i\leq m$ and $1\leq j\leq n$, define $\sigma_{ij}=g^{-1}(\tau_j)\cap\sigma_i$ and define $\tau_{ij}=g(\sigma_{ij})$. Now define two $(m+1)\times (n+1)$ matrices $A$ and $B$ via:
\begin{itemize}
    \item For $1\leq i\leq m$ and $1\leq j\leq n$, $A_{ij}=\operatorname{area}(\sigma_{ij})$ and $B_{ij}=\operatorname{area} (\tau_{ij})$.
    \item For $0\leq j\leq n$, $A_{j0}=0$. 
    \item For $0\leq i\leq m$, $B_{0i}=0$.
    \item For $0\leq i\leq m$, \begin{equation}A_{i0}=a_i-\sum\limits_{j=1}^{n}\operatorname{area}(\sigma_{ij}).\end{equation}
    \item For $0\leq j\leq n$, \begin{equation}B_{0j}=b_j-\sum\limits_{i=1}^{m}\operatorname{area} (\tau_{ij}).\end{equation}
\end{itemize}

\noindent The pair of matrices $(A,B)$ is a discrete semi-coupling from $\psi(S_1)$ to $\psi(S_2)$ by construction. We say that $(A,B)$ is the semi-coupling \textit{ corresponding} to the homeomorphism $g$. Let $1\leq i\leq m$ and $1\leq j\leq n$ and let $M_{ij}=f^{-1}(\sigma_{ij})\subset M$. Denote the area multiplication factor of $g$ on $\sigma_{ij}$ by $m_{ij}$. Then by the Schwarz inequality,
\begin{align}
    \int_{M_{ij}}q_f\cdot q_{g\circ f}&=\int_{M_{ij}}\sqrt{a(x)}u_i\cdot \sqrt{a(x)m_{ij}(x)}v_j\,dx=\int_{M_{ij}}\sqrt{a(x)}\sqrt{a(x)m_{ij}(x)}\,dx(u_i\cdot v_j)\\
    &\leq\sqrt{\int_{M_{ij}}a(x)\,dx  \int_{M_{ij}}a(x)m_{ij}(x)\,dx} (u_i\cdot v_j)=\sqrt{A_{ij}}\sqrt{B_{ij}}(u_i\cdot v_j)
\end{align}
Summing over all $i$ and $j$ we obtain:
\begin{align}
    \int_M q_f\cdot q_{g\circ f}&=\sum_{i,j}\int_{M_{ij}} q_f\cdot q_{g\circ f}=\sum_{i,j}\int_{M_{ij}}a(x)m(x) (u_i\cdot v_j)\leq \sum_{i,j}\sqrt{A_{ij}}\sqrt{B_{ij}}(u_i\cdot v_j).
\end{align}

This completes the proof of Claim \ref{maintheoremclaim3}. It follows that,
\begin{equation}\sup_{g\in\operatorname{Hom}_{PL}^+(S_1,S_2)}\int_M q_f\cdot q_{g\circ f}\leq\sup_{\{A_{ij}\},\{B_{ij}\}\in\mathcal{A}}\sum_{i=1}^{m}\sum_{j=1}^{n}\sqrt{A_{ij}B_{ij}}(u_i\cdot v_j).\end{equation}
and thus the theorem is proved.
\end{proof}

\begin{remark}
In this section we have defined a mapping from the shape space of PL surfaces in $\mathbb R^3$ to the space of finitely supported measures on $S^2$. We have then shown that the SRNF (pseudo-) distance between two surfaces is equal to the WFR distance between the two corresponding measures. It is shown in \cite{chizat2018interpolating,chizat2018unbalanced,liero2018optimal} that the space of finitely supported measures on $S^2$ is a geodesic length space. One might hope that geodesics in this space could somehow be ``lifted" to geodesics in the space of PL surfaces. The main problem with this plan is that there is an infinite-dimensional space of surfaces corresponding to each measure; see \cite{klassen2020closed} for examples of arbitrarily high dimensional spaces of surfaces corresponding to a single measure. Hence, there is no unique way of lifting geodesics in the space of measures to the space of surfaces. Because of this degeneracy, which is inherent to the SRNF, there is no direct way to define a ``geodesic" in the space of surfaces with respect to the SRNF distance function. While it is true that  existing methods (involving gradient searches over the reparametrization group; see for example \cite{jermyn2017elastic}) do result in plausible-looking deformations from one surface to another, these deformations are not geodesics in any strict mathematical sense. Thus in order to define geodesics formally in the space of surfaces, one would have to resolve the degeneracy of the SRNF distance function by, for example, adding extra terms to the SRNF metric, see e.g.~\cite{jermyn2012elastic,su2020shape}.

The mapping we introduced does, however, restrict to a bijection between the space of \textit{ closed convex} PL surfaces in $\mathbb R^3$ and a subspace of the finitely supported measures on $S^2$. Therefore it would be possible to define geodesics in the space of closed convex surfaces as lifts of geodesics in this subspace; this was done recently for the space of convex curves in \cite{charon2020length}. 
\end{remark}

\begin{remark}
Using the mapping that we have defined from the shape space of surfaces in ${\mathbb R}^3$ to measures on $S^2$, we can pull back any distance function on the space of measures to obtain a pseudo-distance function on the shape space of surfaces. While the main purpose of this paper is to show that the SRNF pseudo-distance is obtained in this manner from the WFR distance on the space of measures, it might be of interest to use other distance functions on the space of measures to obtain interesting pseudo-distances on the space of surfaces. A likely candidate for this would be the $\rho$-parametrized WFR distance function defined in \cite{chizat2018unbalanced}. The corresponding pseudo-distance on surfaces would have a natural interpretation as assigning different weights to the direction of the normal vector as opposed to the shrinking or expansion of area on the surfaces. As explained in Section 2, the algorithm introduced in Section 2 would also provide a computation of this generalized version of the SRNF pseudo-distance. Note that any pseudo-distance on surfaces obtained in this way would have the same degeneracy as the SRNF pseudo-distance.
\end{remark}

\subsection{SRNF Computation Experiments}
By Theorem~\ref{thm:PLequiv}, we can utilize Algorithm~\ref{alg:WFR_iter} to compute the exact SRNF pseudo-distance directly between two simplicial meshes. The resulting method is described in Algorithm~\ref{alg:SRNF}.

\begin{algorithm}
\caption{Calculate $\operatorname{SRNF}$ shape distance\label{alg:SRNF}}
\begin{algorithmic} 
\Procedure{SRNF\_distance}{$S_1,S_2,\epsilon$}

\State $u_i\gets$ the unit normal of the $i$th face of $S_1$  
\State $a_i\gets$ the area of the $i$th face of $S_1$ 
\State $v_j\gets$ the unit normal of the $i$th face of $S_2$ 
\State $b_j\gets$ the area of the $i$th face of $S_2$ 
\State \Return \Call{WFR\_distance}{$a,b,u,v,\epsilon$}
\EndProcedure
\end{algorithmic}
\end{algorithm}
 
To quantify the performance of Algorithm~\ref{alg:SRNF} we compare it to the method introduced in ~\cite{laga2017numerical}. In their implementation surfaces are assumed to be represented by a spherical parametrization and the diffeomorphism group is discretized using spherical harmonics. 
This in turn allows one to formulate the SRNF distance computation as a constrained minimization problem over the coefficients of the reparametrization in the chosen spherical harmonics basis. In the following, we will refer to this method as the parametrization-based method.

Most data that one encounters in real applications is, however, not given in such a parametrized form but rather as a simplicial complex. Thus one first has to solve the parametrization problem ~\cite{sheffer2007mesh}, 
which is of comparable difficulty to the geodesic boundary problem itself. In our experiments, we used 24 shapes from the TOSCA dataset (simplicial meshes), for which we also had access to spherical parametrizations, which have been calculated using the method of~\cite{praun2003spherical} as implemented in~\cite{kurtek2013landmark}. We present the result of this comparison in Figure~\ref{fig:Corr} which consists of three subplots: Figure~\ref{fig:Corr} (a), which highlights the distances computed with both methods for 552 pairs of PL surfaces. The precise distances produced by our method (Orange) are consistently lower than the parametrization-based distances produced by the method of ~\cite{laga2017numerical} (Blue). The mean relative error of the parametrization-based method compared to our method is $174.582\%$ with a standard deviation of $90.077\%$. Note that this error consists of both an approximation error of the spherical parameterization of the simplicial complex and an approximation error of the optimal reparametrization.  In Figure~\ref{fig:Corr} (b), we plot the correlation between these two methods of computing the distances which have a Pearson correlation coefficient of 0.793. Note that this is comparable to the correlation of elastic distances between functions on the line that are either computed using dynamic programming or computed using an exact algorithm, cf.~\cite{hartman2021supervised,srivastava2016functional,lahiri2015precise}.

\begin{figure}[ht!]
\centering
\begin{minipage}{.45\textwidth}
\centering
\includegraphics[width=\textwidth]{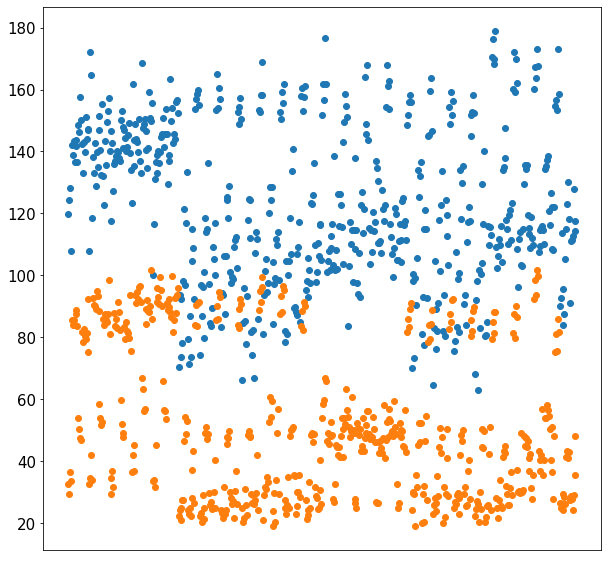}\\
\small(a) SRNF shape distances produced by Algorithm~\ref{alg:SRNF} (Orange) and by the method of ~\cite{laga2017numerical} (Blue). \\
\end{minipage}
\begin{minipage}{.45\textwidth}
\centering
\includegraphics[width=\textwidth]{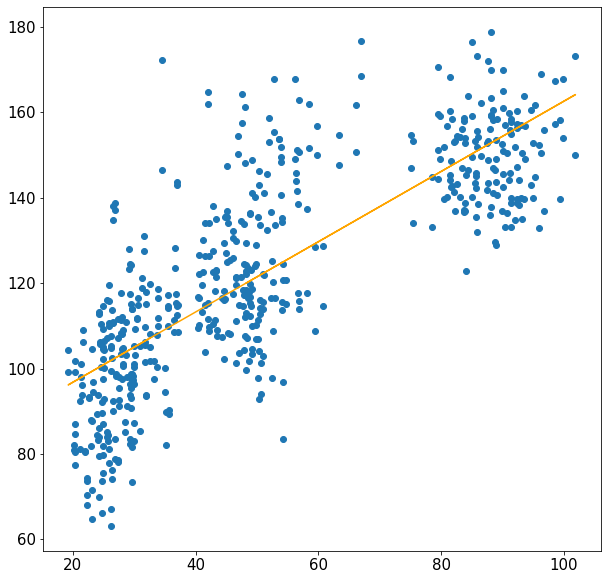}\\
\small(b) Correlation plot of distances from the method of ~\cite{laga2017numerical} and Algorithm~\ref{alg:SRNF}.\vspace{\baselineskip}
\end{minipage}
\caption{Comparison of Algorithm~\ref{alg:SRNF} and the method of ~\cite{laga2017numerical} for 552 pairs of surfaces from the TOSCA dataset.The shape distances used in this experiment are sampled from pairs of shapes from the TOSCA data set as displayed in Figure \ref{fig:MDS}.}
\label{fig:Corr}
\end{figure}

The main drawback of our method is that it does not produce an optimal reparameterization that aligns the two surfaces, i.e., we do not obtain point correspondences between the two meshes.  We can, however, still interpret the information in the optimal semi-coupling to visualize a ``fuzzy" correspondence between the surfaces. This is visualized in Figure~\ref{fig:correspo}:
given two PL surfaces $S_1$ and $S_2$, we color each face of $S_2$ according to its unit normal vector. Further let $(A,B)$ be the optimal semi-coupling between $\psi(S_1)$ and $\psi(S_2)$. We then color each face of $S_1$ with normal vector $u_i$ according to the color of the face of $S_2$ with normal vector $v_k$ where $k=\argmax_j A_{ij}$, i.e., each face on $S_1$ is colored by the same color as the face where most of its mass is transported. Examples of such correspondences are presented in Figure~\ref{fig:correspo}.

\begin{figure}[ht!]
\centering
\begin{tabular}{| c c c c c|c|}
\hline
&&$S_1$&&&$S_2$\\

\includegraphics[width=0.15\textwidth]{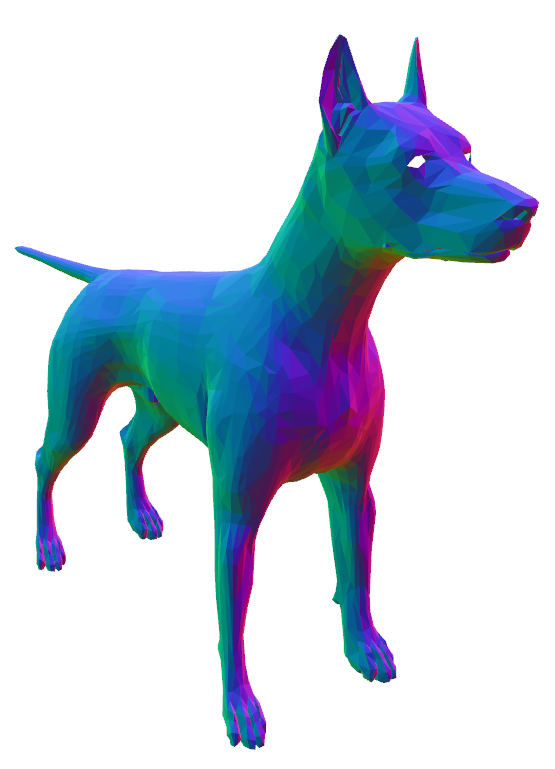}&
\includegraphics[width=0.15\textwidth]{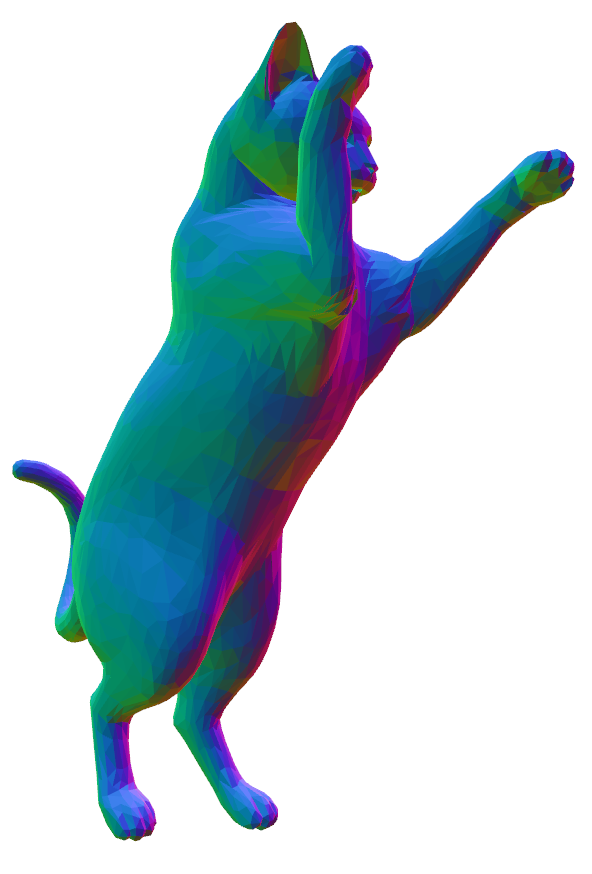} &&
\includegraphics[width=0.15\textwidth]{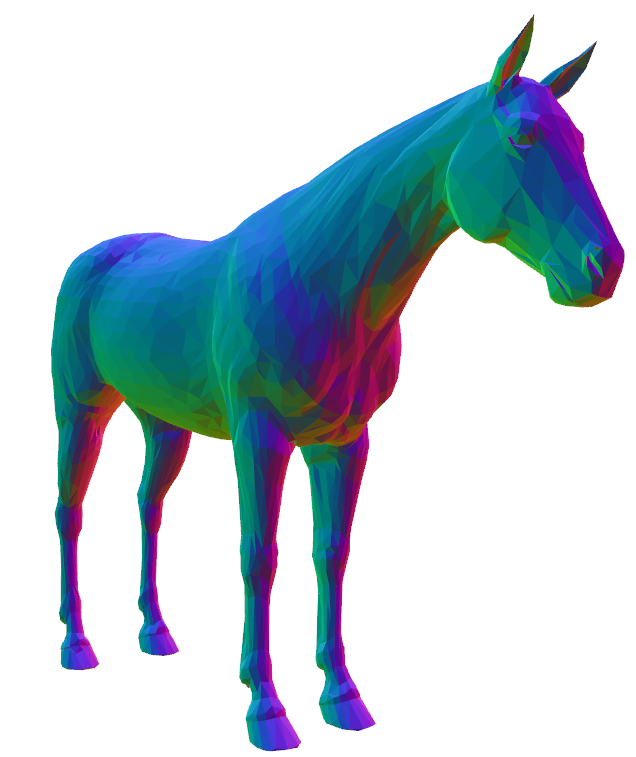}&
\includegraphics[width=0.15\textwidth]{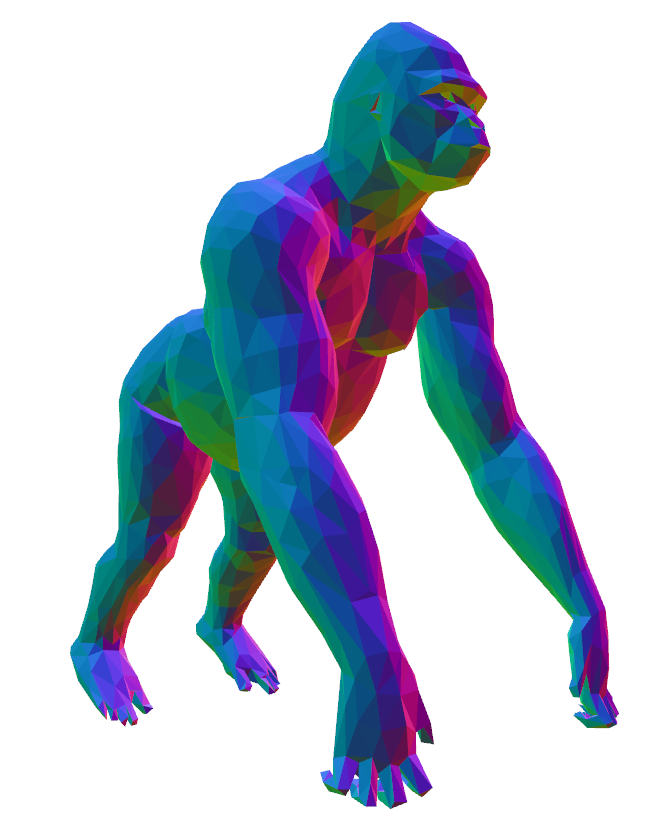}&
\includegraphics[width=0.25\textwidth]{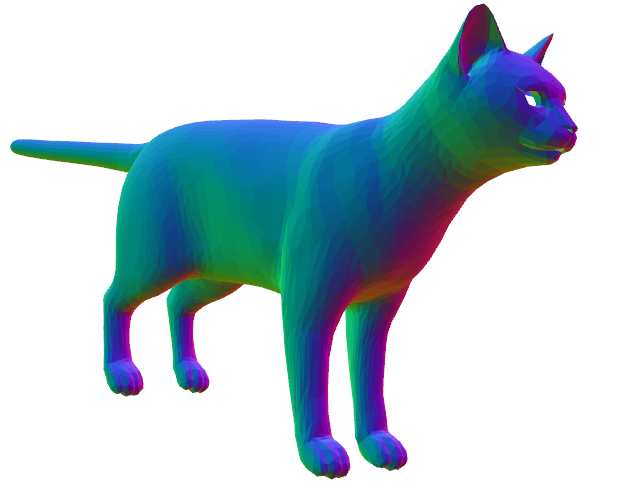} \\
\hline
\end{tabular}

\caption{Fuzzy Point Correspondences: Each face of the source surface ($S_2$) on the right is colored according to its normal vector. The faces of the target surfaces ($S_1$) on the left are colored according to the optimal discrete semi-coupling from $\psi(S_1)$ to $\psi(S_2)$.}\label{fig:correspo}
\end{figure}

The phenomenon of distinct closed surfaces that were indistinguishable by the SRNF shape pseudo-distance was first studied in ~\cite{klassen2020closed}. As a result of Theorem ~\ref{thm:PLequiv}, we obtain a full characterization of this phenomenon: Note that the Wasserstein-Fisher-Rao distance is a true distance and so the SRNF distance between two surfaces is zero if and only if they are mapped to the same measure by $\psi$. A useful observation from the study of convex polyhedra, due to Minkowski, is that every measure on $S^2$ satisfying the closure condition corresponds to a unique (up to translation) closed, convex polyhedron in $\mathbb{R}^3$, see e.g.~\cite{schneider1993convex}. Therefore, each closed PL surface has SRNF distance zero from this unique closed convex polyhedron. In Figure ~\ref{fig:ex_conv}, we give examples of PL surfaces and the convex polyhedron reconstructed from the corresponding measure. This reconstruction is performed using the Python package \textit{polyhedrec}~\cite{sellaroli2017algorithm} available at \href{https://github.com/gsellaroli/polyhedrec}{https://github.com/gsellaroli/polyhedrec}.

\begin{figure}[ht!]
\centering
\begin{tabular}{|C{3.2cm}|C{3.2cm}
|C{3.2cm}|C{3.2cm}|}
\hline
\phantom{h}&&&\\
    \includegraphics[height=0.16
\textwidth]{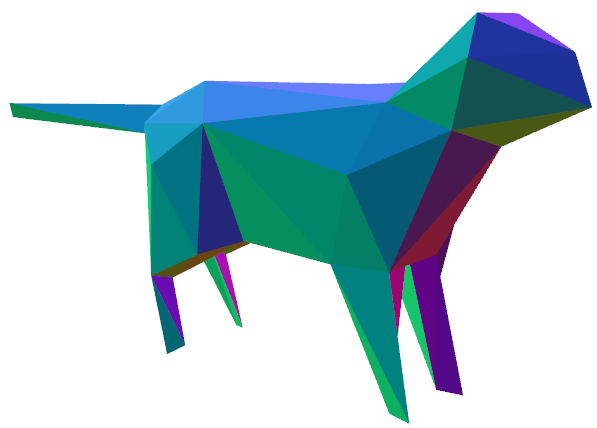}&
    \scalebox{-1}[1]{ \includegraphics[height=0.16
\textwidth]{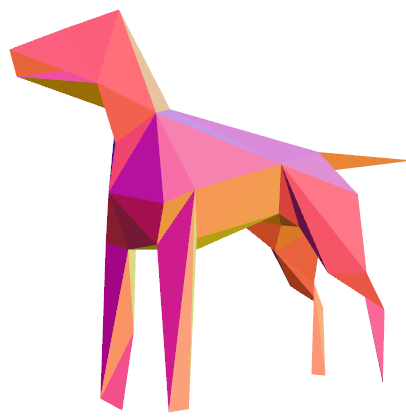}}&
    \includegraphics[height=0.16
\textwidth]{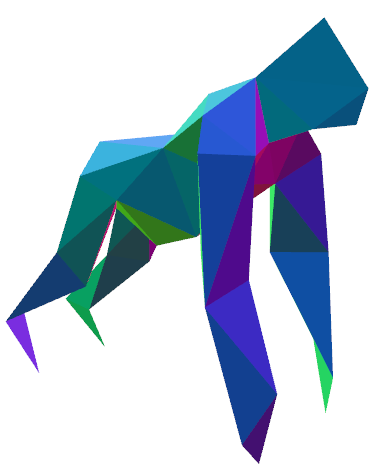}&
    \includegraphics[height=0.16
\textwidth]{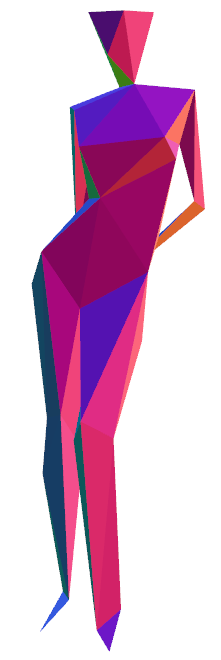}\\
     $\wr$&$\wr$&$\wr$&$\wr$\\
    \includegraphics[height=0.16
\textwidth]{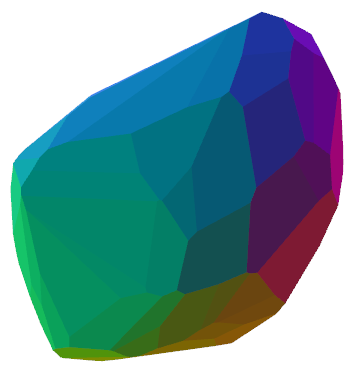}&
    \scalebox{-1}[1]{ \includegraphics[height=0.16
\textwidth]{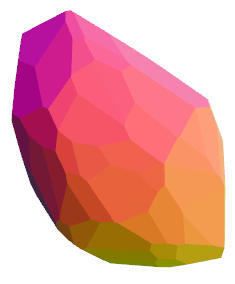}} &
    \includegraphics[height=0.16
\textwidth]{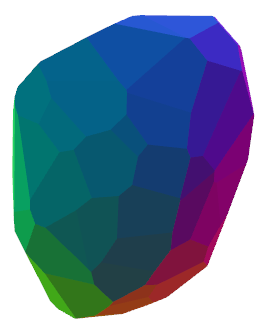} &
    \includegraphics[height=0.16
\textwidth]{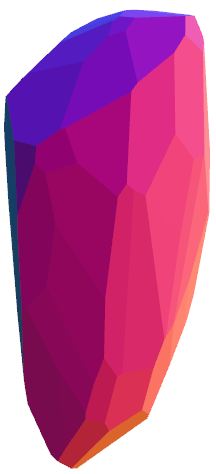}\\
\hline
\end{tabular}
\caption{Four pairs of distinct shapes indistinguishable by the SRNF metric: Each shape on the bottom row is the unique closed, convex polyhedron reconstructed from the area measure associated with the corresponding shape on the top row. Thus the SRNF shape distance between each of these pairs is zero. }
\label{fig:ex_conv}
\end{figure}

Despite these known drawbacks, the SRNF pseudo-metric has been demonstrated to be successful in applications surrounding the classification of surfaces. We demonstrate this by considering a toy example of shapes from the TOSCA dataset, that includes 4 cats, 7 dogs, 17 gorillas, 10 horses, and 9 lionesses. We then compute the SRNF pseudo-distance matrix using our Algorithm~\ref{alg:SRNF} and visualize the results using 3d multi-dimensional scaling in Figure~\ref{fig:MDS}. One can see that the SRNF distance, despite its degeneracy, produces meaningful clusters. 

\begin{figure}[ht!]
\centering
\begin{minipage}{.8\textwidth}
\centering
\includegraphics[width=1\textwidth]{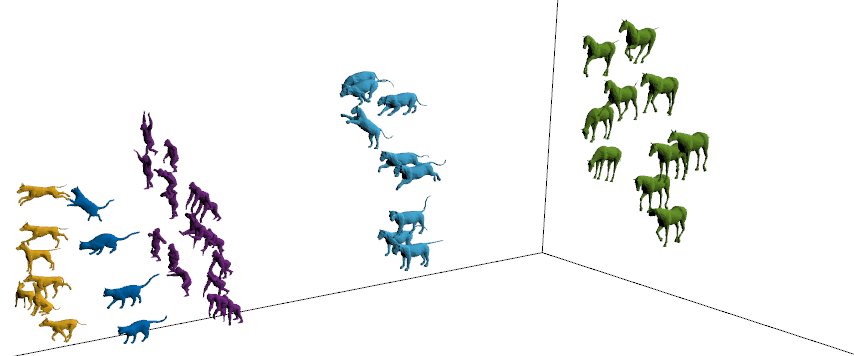}
\end{minipage}
\caption{Multi-Dimensional Scaling Plot: A subset of the TOSCA dataset is plotted according to 3d  multi-dimensional scaling of the SRNF distance matrix produced by Algorithm ~\ref{alg:SRNF}.}
\label{fig:MDS}
\end{figure}

\section{Conclusion}
In this article we propose a novel method to \emph{precisely} compute the SRNF shape distance between PL surfaces. This method follows from three results each of which has interesting implications in its own right:

First, we propose a novel method for computing the Wasserstein-Fisher-Rao distance in unbalanced optimal transport. While fast estimations of this metric can be achieved by including an entropy regularization term, we propose a new method that solves exactly for the WFR distance.

Second, we extend the SRNF framework to surfaces parameterized by Lipschitz immersions. This class of surfaces notably includes both smooth surfaces and PL surfaces. We then show that this extended framework is consistent with the original SRNF distance which was formulated in the smooth category. 

Finally, we establish an equivalence between the SRNF and the Wasserstein-Fisher-Rao distance on the space of Borel measures on $S^2$. In addition to establishing a new method to compute the SRNF shape distance, this result offers insight into theoretical problems that exist surrounding the SRNF shape distance. For instance, this result gives us tools to analyze the phenomenon of distinct closed surfaces that are indistinguishable by the SRNF shape pseudo-distance as highlighted in Figure~\ref{fig:ex_conv}.

\noindent\textbf{Open Questions and Future Work:} This project provides several open questions and opportunities for future work. The first set of open questions concerns Algorithm~\ref{alg:WFR_iter}:
in all of Section~\ref{algorithm}, we consider specifically the Wasserstein-Fisher-Rao distance between measures supported on $S^2$.  An obvious open problem is to investigate whether the results of Section~\ref{algorithm} and the associated algorithm can be generalized to a larger class of domains and general unbalanced OMT problems. Extending this algorithm to more general domains should follow easily from alternate characterizations of the distance developed in \cite{chizat2018scaling,chizat2018unbalanced, liero2018optimal}. Generalizing the algorithm to other unbalanced OMT problems may be more challenging and require restricting to unbalanced OMT problems that only optimize cost functions satisfying certain properties. There are also possibilities for future work in making our algorithm more efficient, by implementing sparse data types and/or absorption thresholds as entries of the discrete semi-coupling approach zero. Further, as future work, we would like to update our implementation to work with more general weight matrices so that our algorithm can be effectively used for other applications of unbalanced optimal transport.

The main result of this paper gives us new tools for studying the SRNF shape distance. Currently, we have only proven this result for PL shapes and we leave the extension of Theorem~\ref{thm:maintheorem} to all surfaces in $\mathcal{S}_{\operatorname{Lip}}$ as an open problem. We expect that this result could follow from Theorem~\ref{thm:maintheorem} using the density of PL surfaces in $\mathcal{S}_{\operatorname{Lip}}$ with respect to the SRNF pseudo-metric and carefully studying the continuity of all involved operations. Further, there are opportunities for future work in characterizing the relationships between shapes that are indistinguishable via the SRNF shape distance. Understanding this relationship may serve to help in developing meaningful SRNF based interpolations between shapes based on optimal discrete semi-couplings between the associated measures.

Another interesting subject for future research is to gain a better understanding of the set of all shapes of closed surfaces corresponding to a given measure on $S^2$. It seems clear that this set is, in general, infinite dimensional. See \cite{klassen2020closed} for a construction of arbitrarily high dimensional sets of this form. Consequently there are several natural questions that arise: Does the geometry of this set depend on the given measure on $S^2$? Can the shape space of surfaces be thought of as a fiber bundle over the space of measures on $S^2$?

\newpage
\bibliographystyle{plain}

\begin{thebibliography}{10}

\bibitem{alexandrov1938theorie}
AD~Alexandrov.
\newblock Zur theorie der gemischten volumina von konvexen k{\"o}rpern i.
\newblock {\em Mat. Sbornik NS}, 1:227--251, 1938.

\bibitem{bauer2014overview}
Martin Bauer, Martins Bruveris, and Peter~W Michor.
\newblock Overview of the geometries of shape spaces and diffeomorphism groups.
\newblock {\em Journal of Mathematical Imaging and Vision}, 50(1):60--97, 2014.

\bibitem{bauer2020numerical}
Martin Bauer, Nicolas Charon, Philipp Harms, and Hsi-Wei Hsieh.
\newblock A numerical framework for elastic surface matching, comparison, and
  interpolation.
\newblock {\em International Journal of Computer Vision}, pages 1--20, 2021.

\bibitem{bauer2011sobolev}
Martin Bauer, Philipp Harms, and Peter~W Michor.
\newblock Sobolev metrics on shape space of surfaces.
\newblock {\em Journal of Geometric Mechanics}, 3(4), 2011.

\bibitem{benamou2000computational}
Jean-David Benamou and Yann Brenier.
\newblock A computational fluid mechanics solution to the {M}onge-{K}antorovich
  mass transfer problem.
\newblock {\em Numerische Mathematik}, 84(3):375--393, 2000.

\bibitem{bertsekas1997nonlinear}
Dimitri~P Bertsekas.
\newblock Nonlinear programming.
\newblock {\em Journal of the Operational Research Society}, 48(3):334--334,
  1997.

\bibitem{bethuel1988density}
Fabrice Bethuel and Xiaomin Zheng.
\newblock Density of smooth functions between two manifolds in {S}obolev
  spaces.
\newblock {\em Journal of functional analysis}, 80(1):60--75, 1988.

\bibitem{bruveris2016optimal}
Martins Bruveris.
\newblock Optimal reparametrizations in the square root velocity framework.
\newblock {\em SIAM Journal on Mathematical Analysis}, 48(6):4335--4354, 2016.

\bibitem{bruveris2014geodesic}
Martins Bruveris, Peter~W Michor, and David Mumford.
\newblock Geodesic completeness for {S}obolev metrics on the space of immersed
  plane curves.
\newblock In {\em Forum of Mathematics, Sigma}, volume~2. Cambridge University
  Press, 2014.

\bibitem{burkard2012assignment}
Rainer Burkard, Mauro Dell'Amico, and Silvano Martello.
\newblock {\em Assignment problems: revised reprint}.
\newblock SIAM, 2012.

\bibitem{charon2020length}
Nicolas Charon and Thomas Pierron.
\newblock On length measures of planar closed curves and the comparison of
  convex shapes.
\newblock {\em Annals of Global Analysis and Geometry}, 60(4):863--901, 2021.

\bibitem{chizat2018interpolating}
Lenaic Chizat, Gabriel Peyr{\'e}, Bernhard Schmitzer, and Fran{\c{c}}ois-Xavier
  Vialard.
\newblock An interpolating distance between optimal transport and
  {F}isher--{R}ao metrics.
\newblock {\em Foundations of Computational Mathematics}, 18(1):1--44, 2018.

\bibitem{chizat2018scaling}
Lenaic Chizat, Gabriel Peyr{\'e}, Bernhard Schmitzer, and Fran{\c{c}}ois-Xavier
  Vialard.
\newblock Scaling algorithms for unbalanced optimal transport problems.
\newblock {\em Mathematics of Computation}, 87(314):2563--2609, 2018.

\bibitem{chizat2018unbalanced}
Lenaic Chizat, Gabriel Peyr{\'e}, Bernhard Schmitzer, and Fran{\c{c}}ois-Xavier
  Vialard.
\newblock Unbalanced optimal transport: Dynamic and {K}antorovich formulations.
\newblock {\em Journal of Functional Analysis}, 274(11):3090--3123, 2018.

\bibitem{cuturi2013sinkhorn}
Marco Cuturi.
\newblock Sinkhorn distances: Lightspeed computation of optimal transport.
\newblock In C.~J.~C. Burges, L.~Bottou, M.~Welling, Z.~Ghahramani, and K.~Q.
  Weinberger, editors, {\em Advances in Neural Information Processing Systems},
  volume~26. Curran Associates, Inc., 2013.

\bibitem{dryden2016statistical}
Ian~L Dryden and Kanti~V Mardia.
\newblock {\em Statistical shape analysis: with applications in R}, volume 995.
\newblock John Wiley \& Sons, 2016.

\bibitem{flamary2021pot}
R{\'e}mi Flamary, Nicolas Courty, Alexandre Gramfort, Mokhtar~Z. Alaya,
  Aur{\'e}lie Boisbunon, Stanislas Chambon, Laetitia Chapel, Adrien Corenflos,
  Kilian Fatras, Nemo Fournier, L{\'e}o Gautheron, Nathalie~T.H. Gayraud,
  Hicham Janati, Alain Rakotomamonjy, Ievgen Redko, Antoine Rolet, Antony
  Schutz, Vivien Seguy, Danica~J. Sutherland, Romain Tavenard, Alexander Tong,
  and Titouan Vayer.
\newblock Pot: Python optimal transport.
\newblock {\em Journal of Machine Learning Research}, 22(78):1--8, 2021.

\bibitem{gallouet2021regularity}
Thomas Gallou{\"e}t, Roberta Ghezzi, and Fran{\c{c}}ois-Xavier Vialard.
\newblock Regularity theory and geometry of unbalanced optimal transport.
\newblock {\em arXiv preprint arXiv:2112.11056}, 2021.

\bibitem{haker2004optimal}
Steven Haker, Lei Zhu, Allen Tannenbaum, and Sigurd Angenent.
\newblock Optimal mass transport for registration and warping.
\newblock {\em International Journal of computer vision}, 60(3):225--240, 2004.

\bibitem{hartman2021supervised}
Emmanuel Hartman, Yashil Sukurdeep, Nicolas Charon, Eric Klassen, and Martin
  Bauer.
\newblock Supervised deep learning of elastic {SRV} distances on the shape
  space of curves.
\newblock {\em Proceedings of the IEEE/CVF Conference on Computer Vision and
  Pattern Recognition Workshops. 2021.}, 2021.

\bibitem{jermyn2012elastic}
Ian~H Jermyn, Sebastian Kurtek, Eric Klassen, and Anuj Srivastava.
\newblock Elastic shape matching of parameterized surfaces using square root
  normal fields.
\newblock In {\em European conference on computer vision}, pages 804--817.
  Springer, 2012.

\bibitem{jermyn2017elastic}
Ian~H Jermyn, Sebastian Kurtek, Hamid Laga, and Anuj Srivastava.
\newblock Elastic shape analysis of three-dimensional objects.
\newblock {\em Synthesis Lectures on Computer Vision}, 12(1):1--185, 2017.

\bibitem{joshi2016surface}
Shantanu~H Joshi, Qian Xie, Sebastian Kurtek, Anuj Srivastava, and Hamid Laga.
\newblock Surface shape morphometry for hippocampal modeling in alzheimer's
  disease.
\newblock In {\em 2016 International Conference on Digital Image Computing:
  Techniques and Applications (DICTA)}, pages 1--8. IEEE, 2016.

\bibitem{klassen2020closed}
Eric Klassen and Peter~W Michor.
\newblock Closed surfaces with different shapes that are indistinguishable by
  the {SRNF}.
\newblock {\em Archivum Mathematicum}, 56(2):107--114, 2020.

\bibitem{kondratyev2016new}
Stanislav Kondratyev, L{\'e}onard Monsaingeon, Dmitry Vorotnikov, et~al.
\newblock A new optimal transport distance on the space of finite radon
  measures.
\newblock {\em Advances in Differential Equations}, 21(11/12):1117--1164, 2016.

\bibitem{kurtek2010novel}
Sebastian Kurtek, Eric Klassen, Zhaohua Ding, and Anuj Srivastava.
\newblock A novel {R}iemannian framework for shape analysis of 3{D} objects.
\newblock In {\em 2010 IEEE computer society conference on computer vision and
  pattern recognition}, pages 1625--1632. IEEE, 2010.

\bibitem{kurtek2011elastic}
Sebastian Kurtek, Eric Klassen, John~C Gore, Zhaohua Ding, and Anuj Srivastava.
\newblock Elastic geodesic paths in shape space of parameterized surfaces.
\newblock {\em IEEE transactions on pattern analysis and machine intelligence},
  34(9):1717--1730, 2011.

\bibitem{kurtek2014statistical}
Sebastian Kurtek, Chafik Samir, and Lemlih Ouchchane.
\newblock Statistical shape model for simulation of realistic endometrial
  tissue.
\newblock In {\em ICPRAM}, pages 421--428, 2014.

\bibitem{kurtek2013landmark}
Sebastian Kurtek, Anuj Srivastava, Eric Klassen, and Hamid Laga.
\newblock Landmark-guided elastic shape analysis of spherically-parameterized
  surfaces.
\newblock In {\em Computer graphics forum}, volume~32, pages 429--438. Wiley
  Online Library, 2013.

\bibitem{laga20183d}
Hamid Laga, Yulan Guo, Hedi Tabia, Robert~B Fisher, and Mohammed Bennamoun.
\newblock {\em 3D Shape analysis: fundamentals, theory, and applications}.
\newblock John Wiley \& Sons, 2018.

\bibitem{laga20214d}
Hamid Laga, Marcel Padilla, Ian~H Jermyn, Sebastian Kurtek, Mohammed Bennamoun,
  and Anuj Srivastava.
\newblock 4d atlas: Statistical analysis of the spatiotemporal variability in
  longitudinal 3{D} shape data.
\newblock {\em arXiv preprint arXiv:2101.09403}, 2021.

\bibitem{laga2017numerical}
Hamid Laga, Qian Xie, Ian~H Jermyn, and Anuj Srivastava.
\newblock Numerical inversion of {SRNF} maps for elastic shape analysis of
  genus-zero surfaces.
\newblock {\em IEEE transactions on pattern analysis and machine intelligence},
  39(12):2451--2464, 2017.

\bibitem{lahiri2015precise}
Sayani Lahiri, Daniel Robinson, and Eric Klassen.
\newblock {Precise matching of PL curves in $\mathbb{R}^N$ in the square root
  velocity framework}.
\newblock {\em Geometry, Imaging and Computing}, 2(3):133--186, 2015.

\bibitem{liero2016optimal}
Matthias Liero, Alexander Mielke, and Giuseppe Savar{\'e}.
\newblock Optimal transport in competition with reaction: The
  {H}ellinger--{K}antorovich distance and geodesic curves.
\newblock {\em SIAM Journal on Mathematical Analysis}, 48(4):2869--2911, 2016.

\bibitem{liero2018optimal}
Matthias Liero, Alexander Mielke, and Giuseppe Savar{\'e}.
\newblock Optimal entropy-transport problems and a new
  {H}ellinger--{K}antorovich distance between positive measures.
\newblock {\em Inventiones mathematicae}, 211(3):969--1117, 2018.

\bibitem{maas2015generalized}
Jan Maas, Martin Rumpf, Carola Sch{\"o}nlieb, and Stefan Simon.
\newblock A generalized model for optimal transport of images including
  dissipation and density modulation.
\newblock {\em ESAIM: Mathematical Modelling and Numerical Analysis},
  49(6):1745--1769, 2015.

\bibitem{matuk2020biomedical}
James Matuk, Shariq Mohammed, Sebastian Kurtek, and Karthik Bharath.
\newblock Biomedical applications of geometric functional data analysis.
\newblock In {\em Handbook of Variational Methods for Nonlinear Geometric
  Data}, pages 675--701. Springer, 2020.

\bibitem{merigot2011multiscale}
Quentin M{\'e}rigot.
\newblock A multiscale approach to optimal transport.
\newblock In {\em Computer Graphics Forum}, volume~30, pages 1583--1592. Wiley
  Online Library, 2011.

\bibitem{michor2005vanishing}
Peter~W Michor and David Mumford.
\newblock Vanishing geodesic distance on spaces of submanifolds and
  diffeomorphisms.
\newblock {\em Documenta Mathematica}, 10:217--245, 2005.

\bibitem{monge1781memoire}
Gaspard Monge.
\newblock M{\'e}moire sur la th{\'e}orie des d{\'e}blais et des remblais.
\newblock {\em Histoire de l'Acad{\'e}mie Royale des Sciences de Paris}, 1781.

\bibitem{pennec2006intrinsic}
Xavier Pennec.
\newblock Intrinsic statistics on riemannian manifolds: Basic tools for
  geometric measurements.
\newblock {\em Journal of Mathematical Imaging and Vision}, 25(1):127--154,
  2006.

\bibitem{peyre2019computational}
Gabriel Peyr{\'e}, Marco Cuturi, et~al.
\newblock Computational optimal transport: With applications to data science.
\newblock {\em Foundations and Trends{\textregistered} in Machine Learning},
  11(5-6):355--607, 2019.

\bibitem{piccoli2014generalized}
Benedetto Piccoli and Francesco Rossi.
\newblock Generalized {W}asserstein distance and its application to transport
  equations with source.
\newblock {\em Archive for Rational Mechanics and Analysis}, 211(1):335--358,
  2014.

\bibitem{praun2003spherical}
Emil Praun and Hugues Hoppe.
\newblock Spherical parametrization and remeshing.
\newblock {\em ACM Transactions on Graphics (TOG)}, 22(3):340--349, 2003.

\bibitem{rubner2000earth}
Yossi Rubner, Carlo Tomasi, and Leonidas~J Guibas.
\newblock The earth mover's distance as a metric for image retrieval.
\newblock {\em International journal of computer vision}, 40(2):99--121, 2000.

\bibitem{rumpf2014geometry}
Martin Rumpf and Max Wardetzky.
\newblock Geometry processing from an elastic perspective.
\newblock {\em GAMM-Mitteilungen}, 37(2):184--216, 2014.

\bibitem{rumpf2015variational}
Martin Rumpf and Benedikt Wirth.
\newblock Variational methods in shape analysis.
\newblock {\em Handbook of Mathematical Methods in Imaging}, 2:1819--1858,
  2015.

\bibitem{schneider1993convex}
Rolf Schneider.
\newblock Convex surfaces, curvature and surface area measures.
\newblock In {\em Handbook of convex geometry}, pages 273--299. Elsevier, 1993.

\bibitem{schneider2014convex}
Rolf Schneider.
\newblock {\em Convex bodies: the Brunn--Minkowski theory}.
\newblock Number 151. Cambridge university press, 2014.

\bibitem{sellaroli2017algorithm}
Giuseppe Sellaroli.
\newblock An algorithm to reconstruct convex polyhedra from their face normals
  and areas.
\newblock {\em arXiv preprint arXiv:1712.00825}, 2017.

\bibitem{sheffer2007mesh}
Alla Sheffer, Emil Praun, Kenneth Rose, et~al.
\newblock Mesh parameterization methods and their applications.
\newblock {\em Foundations and Trends in Computer Graphics and Vision},
  2(2):105--171, 2007.

\bibitem{sinkhorn1964relationship}
Richard Sinkhorn.
\newblock A relationship between arbitrary positive matrices and doubly
  stochastic matrices.
\newblock {\em The annals of mathematical statistics}, 35(2):876--879, 1964.

\bibitem{solomon2015transportation}
Justin Solomon.
\newblock {\em Transportation Techniques for Geometric Data Processing}.
\newblock Stanford University, 2015.

\bibitem{solomon2014wasserstein}
Justin Solomon, Raif Rustamov, Leonidas Guibas, and Adrian Butscher.
\newblock {W}asserstein propagation for semi-supervised learning.
\newblock In {\em International Conference on Machine Learning}, pages
  306--314. PMLR, 2014.

\bibitem{srivastava2016functional}
Anuj Srivastava and Eric~P Klassen.
\newblock {\em Functional and shape data analysis}, volume~1.
\newblock Springer, 2016.

\bibitem{su2020shape}
Zhe Su, Martin Bauer, Stephen~C Preston, Hamid Laga, and Eric Klassen.
\newblock Shape analysis of surfaces using general elastic metrics.
\newblock {\em Journal of Mathematical Imaging and Vision}, 62:1087--1106,
  2020.

\bibitem{van2014approximation}
Jean Van~Schaftingen.
\newblock Approximation in {S}obolev spaces by piecewise affine interpolation.
\newblock {\em Journal of Mathematical Analysis and Applications},
  420(1):40--47, 2014.

\bibitem{villani2003topics}
C{\'e}dric Villani.
\newblock {\em Topics in optimal transportation}.
\newblock Number~58. American Mathematical Soc., 2003.

\bibitem{villani2008optimal}
C{\'e}dric Villani.
\newblock {\em Optimal transport: old and new}, volume 338.
\newblock Springer Science \& Business Media, 2008.

\bibitem{whitehead1940c1}
John Henry~C Whitehead.
\newblock {On {C}1-complexes}.
\newblock {\em Annals of Mathematics}, pages 809--824, 1940.

\bibitem{younes1998computable}
Laurent Younes.
\newblock Computable elastic distances between shapes.
\newblock {\em SIAM Journal on Applied Mathematics}, 58(2):565--586, 1998.

\bibitem{younes2010shapes}
Laurent Younes.
\newblock {\em Shapes and diffeomorphisms}, volume 171.
\newblock Springer, 2010.

\end{thebibliography}

\end{document}